\documentclass{amsart}

\usepackage{amsmath}
\usepackage{amsthm}
\usepackage{amssymb}
\usepackage{mathrsfs}
\usepackage{url}
\usepackage{hyperref}
\usepackage[usenames, dvipsnames]{color}
\usepackage{scalerel}
\usepackage{paralist}

\usepackage[]{lineno}

\usepackage{enumitem}
\setenumerate{label={\normalfont(\textit{\roman*})},ref={\textrm{\roman*}}}

\newtheorem{theorem}{Theorem}[section]
\newtheorem*{convention*}{Standing assumption}
\newtheorem{proposition}[theorem]{Proposition}
\newtheorem{lemma}[theorem]{Lemma}
\newtheorem{corollary}[theorem]{Corollary}

\theoremstyle{definition}
\newtheorem{definition}[theorem]{Definition}
\newtheorem{example}[theorem]{Example}
\newtheorem{remark}[theorem]{Remark}
\newtheorem{question}{Question}

\renewcommand{\a}{\alpha}
\renewcommand{\b}{\beta}
\newcommand{\e}{\varepsilon}

\newcommand{\cR}{\mathscr{R}}
\newcommand{\sU}{\mathscr{U}}
\newcommand{\cU}{\mathcal{U}}

\newcommand{\cC}{\mathscr{C}}
\newcommand{\cA}{\mathscr{A}}
\newcommand{\cB}{\mathscr{B}}

\newcommand{\cD}{\mathscr{D}}
\newcommand{\cP}{\mathcal{P}}
\newcommand{\cF}{\mathcal{F}}

\newcommand{\ZZ}{\mathbb{Z}}
\newcommand{\NN}{\mathbb{N}}
\newcommand{\RR}{\mathbb{R}}

\newcommand{\PP}{\mathbb{P}}

\newcommand{\bra}[1]{{\left(#1\right)}}
\newcommand{\brabig}[1]{\big( #1 \big)}
\newcommand{\braBig}[1]{\Big( #1 \Big)}
\newcommand{\bbra}[1]{ { \left\{ #1 \right\} } } 
\newcommand{\abs}[1]{\left|#1\right|}
\newcommand{\norm}[1]{\left\lVert #1 \right\rVert}
\newcommand{\ceil}[1]{\left\lceil #1 \right\rceil}
\newcommand{\floor}[1]{\left\lfloor #1 \right\rfloor}

\newcommand{\ffrac}[2]{#1/#2}

\DeclareMathOperator*{\EE}{{\mathbb{E}}}

\DeclareMathOperator{\id}{id}
\DeclareMathOperator{\Sym}{Sym}

\newcommand{\bb}
{\mathbf}

\newcommand{\cFm}{{ \mathcal{F}_{\! \curlywedge} }}
\newcommand{\uni}{{\mathrm{uni}}}
\newcommand{\f}{{f}}

\begin{document}

\author{Jakub Konieczny}
\title{On consecutive sums in permutations}

\begin{abstract}
 	We study the number of values taken by the sums $\sum_{i=u}^{v-1} a_i$, where $a_1,a_2,\dots,a_n$ is a permutation of $1,2,\dots,n$ and $1 \leq u < v \leq n+1$. In particular, we show that for a random choice of a permutation, with high probability there are $\bra{ \ffrac{(1+e^{-2})}{4} +o(1)} n^2$ such sums. This answers an old question of Erd\H{o}s and Harzheim. We also obtain non-trivial bounds on the maximum possible number of distinct sums, ranging over all permutations of $1,2,\dots,n$. We close with some questions concerning the minimal possible number of distinct sums.
\end{abstract}

\maketitle

\section{Introduction}

\subsection{} For an integer-valued sequence $a = (a_i)_{i=1}^n $, we let $S(a)$ denote the set of all sums of consecutive terms of $a$:
\[ S(a) = \left\{ \sum_{i = u}^{v-1} a_i \ : \ u,v \in \NN, \ 1 \leq u < v \leq n+1\right\}.\]
In this paper, we investigate the number of distinct consecutive sums, that is, the cardinality of $S(a)$. We are primarily interested in the case where $a \in \Sym([n])$, that is, $a$ is a permutation of the set $[n] := \{1,2,\dots,n\}$ for some integer $n \geq 1$. A trivial upper bound $\abs{S(a)} \leq \binom {n+1}{2}$ follows from counting the number of choices of $u$ and $v$ (or, incidentally, from computing $\max S(a) = \sum_{i=1}^n i$). Shorter sequences with maximal number of distinct consecutive sums were investigated by Hegyv{\'a}ri \cite{Hegyvari}, see subsection \ref{ssec:Hegyvari} for details. Distinct sums of pairs of consecutive entries were studied by Freud \cite{Freud-1983}.

Erd\H{o}s and Harzheim considered $\abs{S(a)}$ in special case where $a = \id_n$, meaning that $a_i = i$ for all $1 \leq i \leq n$ \cite{Erdos-1977-27}. They observed that  
\begin{equation}\label{eq:S=o(n2)}
|S(\id_n)|/n^2 \to 0 \text{ as } n \to \infty.
\end{equation}
	 Motivated by \eqref{eq:S=o(n2)}, Erd\H{o}s asked \cite{Erdos-1977-27} if the analogue of \eqref{eq:S=o(n2)} remains true for all permutations.
\begin{question}\label{question:Erdos-Harzheim}
	Is it true that for any $\e > 0$ there exists $n_0 = n_0(\e)$ such that for any $n \geq n_0$ and for any $a \in \Sym([n])$ we have $\abs{S(a)} \leq \e n^2$?
\end{question}	 

\subsection{} We remark that, because of the elementary formula 
\begin{equation}\label{eq:elementary}
\sum_{i=u}^{v-1} i = \frac{1}{2}(v-u)(v+u-1),
\end{equation}
the problem of estimating $\abs{S(\id_n)}$ is closely connected to the problem of estimating $\abs{ [n] \cdot [n]}$, which is known as the \emph{multiplication table problem} and has been extensively studied. Here, we use the standard notation $A \cdot B = \{ab \ : \ a \in A,\ b \in B\}$. To make this connection explicit, note that \eqref{eq:elementary} implies that ${S(\id_n)} \subset [2n] \cdot [2n]$. Conversely, an elementary computation shows that $S(\id_n) \supset \big[\floor{n/2}\big] \cdot \big[\floor{n/2}\big] \cap (2\NN+1)$. 

The first proof that $\abs{[n] \cdot [n] }/n^2 \to 0$ as $n \to \infty$ is due to Erd\H{o}s \cite{Erdos-1955}, with further quantitative improvements  by the same author \cite{Erdos-1960}. The exact asymptotics were obtained by Ford \cite{Ford-2008}, who showed that
\begin{equation}
\abs{ [n] \cdot [n] } = \Theta\bra{ n^2 (\log n)^{-\mathcal{E}} (\log \log n)^{-3/2}},
\end{equation}
	 where $\displaystyle \mathcal{E}=1-\frac{{1+\log \log 2}}{\log 2} = 0.086\dots$. A later result of the same author \cite{Ford-2019} shows that the same asymptotic remain valid if we restrict to odd integers:
\begin{equation}
\abs{ [n] \cdot [n] \cap (2\NN+1)} = \Theta\bra{ n^2 (\log n)^{-\mathcal{E}} (\log \log n)^{-3/2}}.
\end{equation}
It follows that
\begin{equation}
|S(\id_n)| = \Theta\bra{ n^2 (\log n)^{-\mathcal{E}} (\log \log n)^{-3/2}}.
\end{equation}

\subsection{} The first goal of this paper is to show that the answer to Question \ref{question:Erdos-Harzheim} is an emphatic \emph{``no''}. Without further ado, we present a simple counterexample.

\begin{proposition}\label{I:prop:cexple-basic}
	For any $n \geq 1$ there exists $a \in \Sym([n])$ such that $\abs{S(a)} \geq \frac{1}{4}n^2. $
\end{proposition}
\begin{proof}
	Let $a \in \Sym([n])$ be the permutation $1,n,2,n-1,3,n-2,\dots$, that is, $a_{i} = (i+1)/2$ if $i$ is odd and $a_{i} = n+1-i/2$ if $i$ is even. In particular, for each odd $i$ with $1 \leq i < n$ we have $a_i + a_{i+1} = n+1$. 
	
	Let $\widetilde S \subset S(a)$ be the set of the consecutive sums with odd length, that is
\[ \tilde S = \left\{ \sum_{i = u}^{v-1} a_i \ : \ 1 \leq u < v \leq n+1,\ v-u \equiv 1 \bmod 2 \right\}.\]
	We will show that each $s \in \tilde S$ has a unique representation as $s = \sum_{i = u}^{v-1} a_i$ where $1 \leq u < v \leq n+1$ and $v-u \equiv 1 \bmod 2$, and hence 
	\[\abs{S(a)} \geq |\tilde S| =  \ceil{\frac{n+1}{2} }\cdot \floor{ \frac{n+1}{2} } \geq \frac{n^2}{4}.\]
Pick any $s \in \tilde S$. For each representation $s = \sum_{i = u}^{v-1} a_i$  ($1 \leq u < v \leq n+1,\ v-u \equiv 1 \bmod 2$) we have $s = (n+1)l + k$, where $l = (v-u-1)/2$ and $k = a_{v-1}$ if $u \equiv 1 \pmod{2}$ and $k = a_u$ if $u \equiv 0 \pmod{2}$. Suppose that $s = \sum_{i = u'}^{v'-1} a_i$ was another such representation, and let $k',l'$ be defined accordingly. Then $l' = l = \floor{s/(n+1)}$, and hence also $k = k'$. Let $w$ be the integer with $1 \leq w \leq n$ and $a_w = k$. Then $u' \equiv u \equiv w \bmod 2$. Hence, either $u' = u = w$ or $v'-1 = v-1 = w$. In either case, it follows that $u' = u$ and $v' = v$.  
\end{proof}

The constant $\frac{1}{4}$ in Proposition \ref{I:prop:cexple-basic} can be improved using a randomised variant of the construction above. Likewise, the constant $\frac{1}{2}$ in the trivial upper bound $\abs{S(a)} \leq (\frac{1}{2} + o(1))n^2$ can be improved using a somewhat more involved optimization argument, which is perhaps the most novel contribution in this paper. Here and elsewhere, $o(1)$ is shorthand for a quantity which tends to $0$ as $n \to \infty$.

\begin{theorem}\label{thm:INT:A} \label{I:prop:cexple-improved} \label{I:prop:upper-bound}
	Let $n \geq 1$ be an integer. Then 
	\begin{equation}\label{eq:INT:010}
	(c_1+o(1)) n^2 \leq \max_{a \in \Sym([n])} \abs{S(a)} \leq (c_2+o(1))n^2 
	\end{equation}	 
	where $\displaystyle c_1 = \frac{3}{2}-\frac{2}{\sqrt{e}} = 0.286\dots$ and $\displaystyle c_2 = \frac{1}{4} + \frac{\pi}{16} = 0.446\dots$.
\end{theorem}

The upper and lower bound in \eqref{eq:INT:010} are proved in Sections \ref{section:LOW} and \ref{section:EXT-up} respectively. It would be surprising if either of the constants $c_1,c_2$ in Theorem \ref{thm:INT:A} was optimal. However, we expect that the maximal number of distinct consecutive sums in a permutation of a given size should have somewhat regular asymptotics, which prompts us to pose the following question.

\begin{question}
	Does there exist a constant $c > 0$ such that for all $n \geq 1$ we have 	
	\begin{equation}\label{eq:INT:011}
	\max_{a \in \Sym([n])} \abs{S(a)} = (c+o(1))n^2,
	\end{equation}	
	and if so, what is the value of $c$?
\end{question}

\subsection{} While the results mentioned above answer the original question of Erd\H{o}s, they do not say what happens for a ``typical'' permutation. Our next result shows that the answer to the Question \ref{question:Erdos-Harzheim} is still negative ``on average'', in a rather strong sense.

\begin{theorem}\label{thm:INT:B}
	Let $n \geq 1$ be an integer and let $\bb a$ be a permutation of $[n]$ chosen uniformly at random. Put $\displaystyle c= \frac{{1+e^{-2}}}{4} = 0.283\dots$. Then, for each $\delta > 0$,
	\begin{equation}\label{eq:INT:700}
\PP\bra{ \abs{ \abs{S(\bb a)} - cn^2 } > \delta n^2} = o(1).
	\end{equation}	
\end{theorem}

The conclusion of Theorem \ref{thm:INT:B} can be more briefly stated by saying that $\abs{S(\bb a)}/n^2$ converges to $c$ in probability as $n \to \infty$. In particular, since $\abs{S(\bb a)}/n^2$ is bounded, it follows that
	\begin{equation}\label{eq:INT:701}
\displaystyle \EE(\abs{S(\bb a)}) = (c + o(1)) n^2.
	\end{equation}	
The proof of Theorem \ref{thm:INT:B} is carried out in Section \ref{section:EXC}, dealing with the expected value of $\abs{S(\bb a)}$, and  Section \ref{section:HM}, dealing with the second moment $\EE \abs{S(\bb a)}^2$.

\subsection{}\label{ssec:Hegyvari} To close this section, we remark that a similar problem was investigated by Hegyv{\'a}ri in \cite{Hegyvari}. Instead of a permutation of $[n]$, Hegyv{\'a}ri considers shorter sequences $a = (a_i)_{i = 1}^k$ where $k \leq n$ and $a_i \in [n]$ for all $i$. He then investigates the largest value of $k$, say $k_{\max}(n)$, for which there exists a sequence $a$ of length $k$, like above, such that \emph{all} consecutive sums $\sum_{i=u}^{v-1} a_i$ are distinct ($1 \leq u < v \leq k+1)$.
 He shows that
\begin{equation}\label{eq:INT:720}
	(1/3+o(1))n \leq k_{\max}(n) \leq (2/3+o(1))n.
\end{equation}
Note that any sequence taking values in $[n]$ that has distinct consecutive sums in particular has no repeated entries, and hence can be extended to a permutation. It follows that
\[
	\max_{b \in \Sym([n])} \abs{S(a)} \geq \binom{k_{\max}(n)+1}{2} \geq \bra{\frac{1}{18}+o(1)}n^2,
\]
hence Hegyv{\'a}ri's result also yields negative answer to Question \ref{question:Erdos-Harzheim} and an analogue of Proposition \ref{I:prop:cexple-basic} with a slightly worse constant $\ffrac{1}{18}$ in place of $\ffrac{1}{4}$. Conversely, the upper bound in Theorem \ref{thm:INT:A} implies the bound 
\[ k_{\max}(n) \leq \bra{\sqrt{\ffrac{\pi}{8} + \ffrac{1}{2}} + o(1)}n = (0.944\dots + o(1)) n,\]
which is non-trivial but significantly worse than Hegyv{\'a}ri's upper bound.

\subsection*{Notation} 
We let $\NN = \{1,2,\dots\}$ and $\NN_0 = \NN \cup \{0\}$. For $n \geq 1$, $[n] = \{1,2,\dots,n\}$. For a sequence $a =(a_i)_{i\in I}$ and a set $J \subset I$, $a|_J = (a_j)_{j \in J}$ denotes the restriction of $a$ to $J$. We use the standard shorthand $e(t) = e^{2\pi i t}$ for $t \in \RR$. We consistently use the convention where random variables are denoted with a boldface font.

Throughout the paper, we work in the regime $n \to \infty$. For two expressions $X = X(n)$ and $Y=Y(n)$  we write $X = O(Y)$ if there exists a constant $C > 0$ (independent of $n$) such that $\abs{X} \leq C Y$. If the constant $C$ is allowed to depend on a parameter $M$, we write $X = O_M(Y)$. Likewise, we write $X = \Omega(Y)$ if there exists a constant $c > 0$ such that $\abs{X} \geq c Y$. If $X = O(Y)$ and $X = \Omega(Y)$, we write $X = \Theta(Y)$. We write $X = o(Y)$ if $Y > 0$ for $n$ large enough and $\ffrac XY \to 0$ as $n \to \infty$, and similarly $X = \omega(Y)$ if $X > 0$ for $n$ large enough and $\ffrac XY \to \infty$ as $n \to \infty$.
Expressions such as $O(X)$ and $o(X)$ are often used to denote unspecified functions with the asymptotic behaviour as just described. For instance, $o(1)$ and $\omega(1)$ denote quantities which tend to $0$ and $+\infty$ respectively as $n \to \infty$.

\subsection*{Acknowledgements} The author wishes to thank Ben Green to pointing out this problem, and for much advice during the work on it. The author is also grateful to Sean Eberhard, Freddie Manners,  Przemek Mazur, Rudi Mrazovi\'{c} and Aled Walker for many fruitful discussions. Finally, the author thanks Christian Elsholtz, 
Fan Chung, Norbert Hegyv\'{a}ri and Jozsef Solymosi for helpful comments.

The author acknowledges the generous
support from the Clarendon Fund and SJC Kendrew Fund during the work on this paper. During the revision process of this paper the author also received funding from ERC grant ErgComNum 682150, Foundation for Polish Science (FNP) and l'Universit\'{e} de Lyon. 
\section{Average number of sums}\label{section:EXC}

\subsection{}

In this section we study the asymptotic behaviour of $\EE\abs{S(\bb a)}$ where $\bb a$ is a random permutation of $[n]$ and $n \to \infty$. Throughout, $n \geq 1$ denotes an integer, $\bb a$ denotes an element of $\Sym([n])$ chosen uniformly at random, and all instances of $o(\cdot)$ notation correspond to the limit $n \to \infty$. 

\begin{proposition}\label{EXP:prop:main}
	Let $n$ and $\bb a$ be as introduced above. Then
\begin{equation}\label{eq:THM-A}	
	\EE(\abs{S(\bb a)}) = (c + o(1)) n^2
\end{equation}
where $c = \ffrac{(1+e^{-2})}{4}$. 
\end{proposition}

Much of the work done here will be reused in Section \ref{section:HM}, and hence some of the results are stat ed in a stronger form than immediately necessary.

\subsection{}

One of our basic tools is the Hoeffding inequality. We will mostly use the slightly less well-known variant of it, pertaining to random variables sampled from a finite set without replacement.
 
\begin{theorem}[Hoeffding \cite{Hoeffding-1963}]\label{EXC:thm:Hoeffding}
	Let $\alpha, \beta \in \RR$ and let $\bb x_1, \bb x_2, \dots, \bb x_k$ be a sequence of random variables taking values in $[\alpha,\beta]$. Put $\bb x = \sum_{i=1}^k \bb x_i$ and $\mu = \EE \bb x$. Suppose that either \begin{inparaenum}[(i)]
\item $\bb x_1, \bb x_2, \dots, \bb x_k$ are independent, or 
\item $\bb x_1, \bb x_2, \dots, \bb x_k$ are sampled without replacement from a finite set.
\end{inparaenum}
Then for each $t > 0$ it holds that
	\begin{equation}\label{EXP:eq:01}
	\PP\bra{ \abs{\bb x - \mu } \geq t } \leq 2\exp\bra{- \frac{2 t^2}{k(\b-\a)^2}}.
	\end{equation}
\end{theorem}

As the first application, we record a concentration inequality for the sums $\sum_{i=u}^{v-1} \bb a_i$, where $1 \leq u < v \leq n+1$. Note that the following bound is non-vacuous as soon as the parameter $\tau$ is significantly larger than $\sqrt{\log n}$. 

\begin{corollary}\label{EXC:cor:Hoeffding}
	Let $n \geq 1$ be an integer and let $\bb a$ be a random element of $\Sym([n])$, as introduced above. Then for any $\tau > 0$, the probability that there exist $u,v$ with $1 \leq u < v \leq n+1$ such that 
	\begin{equation}\label{eq:288:1}
	\abs{ \sum_{i=u}^{v-1} \bb a_i - \frac{(v-u)(n+1)}{2} } \geq \tau n  \sqrt{v-u}
	\end{equation}
is less than or equal to $2 n^2 \exp(-2\tau^2 )$.
\end{corollary}
\begin{proof}
By Theorem \ref{EXC:thm:Hoeffding}, for each pair $u,v$ with $1 \leq u < v \leq n+1$, the probability that \eqref{eq:288:1} holds is at most $2 \exp(-2\tau^2 )$. It remains to notice that the number of such pairs $u,v$ does not exceed $n^2$, and apply the union bound.
\end{proof}

\subsection{}\label{SSec:proof-of-thm-B}

In order to estimate $\EE \abs{S(\bb a)}$, we will use the basic identity
\[
	\EE \abs{S(\bb a)} = \sum_{s=1}^{\binom{n+1}{2}} \PP( s \in S(\bb a)
\]
and estimate  the probabilities $\PP( s \in S(\bb a))$ separately for different values of $s$ with $1 \leq s \leq \binom{n+1}{2}$. For technical reasons, we will restrict our attention to consecutive sums that start at positions which are not too close to the origin and also not too far to the right; this is made precise in Proposition \ref{EXC:prop:P-s-in-S} below. To keep track of these restrictions, we introduce a parameter $\e > 0$; because it will appear at various points in the argument we make the following global definition. For the sake of concreteness we could have made a specific choice, such as $\e(n) = 1/\log n$, but we believe that not making this choice improves the exposition.
\begin{convention*}
	Throughout this section $\e = \e(n) > 0$ denotes a positive quantity (dependent on $n$) with $\e = n^{-o(1)}$ and $\e = o(1)$; i.e. $\frac{\log 1/\e(n)}{\log n} \to 0$ and $\e(n) \to 0$ as $n \to \infty$. All implicit error terms are allowed to depend on the choice of $\e$.
\end{convention*}

Given an index $u$ with $1 \leq u \leq n$ and a permutation $a \in \Sym([n])$, we let $S_u(a)$ denote the set of sums starting at $u$, that is
\begin{equation}\label{eq:EXC:001}
S_u( a) := \left\{ \sum_{i=u}^{v-1} a_i \ : \ u < v \leq n+1\right\}.
\end{equation}
Note that one expects, at least heuristically, that the sum $s = \sum_{i=u}^{v-1} a_i$ should be close to to $(u-v)(n+1)/2$, and hence $v$ should be approximately $u + 2s/(n+1)$. Motivated by this observation, we define the set of restricted sums $S'(a)$ by declaring for an integer $s = \sigma \binom{n+1}{2}$, $0 < \sigma \leq 1$, that 
\begin{equation}\label{eq:EXC:002}
s \in S'(a) \iff
s \in S_u(a) \text{ for some } u \text{ s.t. } \e n \leq u \leq (1-\sigma-\e)n.
\end{equation}
We stress that this definition depends on the parameter $\e$ introduced above.

\begin{proposition}\label{EXC:prop:P-s-in-S}
	Let $n$, $\bb a$ and $\e$ be as introduced above. Let $s$ be an integer with $\e \binom{n+1}{2} \leq s \leq (1-\e)\binom{n+1}{2}$ and put $\sigma = s/\binom{n+1}{2}$. Then 
	\begin{equation}\label{EXC:eq:P-s-in-S}
	\PP\Big( s \not \in S'(\bb a)  \Big) = e^{-2+2\sigma} + o(1),
	\end{equation}
	where the error term is uniform with respect to $s$ (but may depend on the definition of $\e$).
\end{proposition}

\begin{proof}[Proof of Proposition \ref{EXP:prop:main}, assuming Proposition \ref{EXC:prop:P-s-in-S}]
	Summing \eqref{EXC:eq:P-s-in-S} over all integers $s$ with $\e \binom{n+1}{2} \leq s \leq (1-\e)\binom{n+1}{2}$ and applying Riemann approximation we obtain
	
	\begin{align*}
		\EE\abs{S'(\bb a)} &= \sum_{s} \PP(s \in S'(\bb a)) + O(\e n^2) \\
		&= \binom{n+1}{2} \int_{\e}^{1-\e} (1-e^{-2+2\sigma}) d \sigma + o( n^2) 
		\\ &= \bra{ \frac{1+e^{-2}}{4} + o(1)} n^2.
	\end{align*}
	
	We next estimate $\abs{S(\bb a) \setminus S'(\bb a) }$. For any $a \in \Sym([n])$ and any sum $s = \sigma \binom{n+1}{2} = \sum_{i=u}^{v-1}  a_i$ in $S( a) \setminus S'( a)$, either $u \leq \e n$ or $u > (1-\sigma - \e)n$. The former possibility accounts for $O(\e n^2)$ elements of $S( a) \setminus S'( a)$. In the latter case, either $v > (1-2\e)n$ (which again accounts for a contribution of $O(\e n^2)$) or $v- u < (\sigma - \e)n$, which implies that $s > (v-u)(n+1)/2 + \e \binom{n+1}{2}$. 
	By Corollary \ref{EXC:cor:Hoeffding}, the probability that there exist $u,v$ with $1 \leq u < v \leq n+1$ such that $\sum_{i=u}^{v-1} \bb a_i - \ffrac{(v-u)(n+1)}{2} > \e \binom{n+1}{2}$ is at most $2 n^2 \exp\bra{-\e^2 n/2}$. Hence, using the trivial bound $\abs{S(a)} \leq \binom{n+1}{2}$ for all $a \in \Sym([n])$, we obtain 
\[
	\EE \abs{S(\bb a) \setminus S'(\bb a) } = O(\e n^2) + O(n^4 \exp\bra{-\e^2 n/2}) = o(n^2),
\]
and \eqref{EXC:eq:P-s-in-S} follows from previous considerations.
\end{proof}

\subsection{}\label{SSec:inclusion-exclusion}

We devote the remainder of this section to proving Proposition \ref{EXC:prop:P-s-in-S}. To this end, we will use the truncated version of the inclusion--exclusion principle, also known as the Bonferroni inequalities. 

Let $1 \leq s \leq \binom{n+1}{2}$ be an integer, and put $\sigma = s/ \binom{n+1}{2}$, $u_{\min} = \ceil{\e n}$, $u_{\max} = \floor{(1-\e - \sigma)n}$. It follows from the inclusion-exclusion principle asserts that

\[ \PP( s \not \in S'(\bb a)) = \PP\left(\neg \bigvee_{u = u_{\min}}^{u_{\max}} s \in S'(\bb a)\right)
 = \sum_{M=0}^{\infty} (-1)^M \sum_{\abs{U} = M}  \PP\bra{ \bigwedge_{u \in U} s \in S_u(\bb a)},
\]
where the inner sum is taken over all sets of integers $U \subset [u_{\min},u_{\max}]$ with $\abs{U} = M$, and the probability of the empty conjunction is $1$ by convention.

Let $N \geq 0$ be an integer. If $N$ is even then it follows from the Bonferroni inequalities that
\begin{equation}\label{EXC:eq:39}
	\PP( s \not \in S'(\bb a)) \leq \sum_{M=0}^{N} (-1)^M \sum_{\abs{U} = M}  \PP\bra{ \bigwedge_{u \in U} s \in S_u(\bb a)},
\end{equation}
where the inner sum is again taken over all sets $U \subset [u_{\min},u_{\max}]$ with $\abs{U} = M$. Conversely, if $N$ is odd then the inequality is reversed:
\begin{equation}\label{EXC:eq:39x}
	\PP( s \not \in S'(\bb a)) \geq \sum_{M=0}^{N} (-1)^M\sum_{\abs{U} = M} \PP\bra{ \bigwedge_{u \in U} s \in S_u(\bb a)}.
\end{equation}
In order to prove Proposition \ref{EXC:prop:P-s-in-S} we will obtain a estimate each of the sums 
$\sum_{\abs{U} = M}\PP\bra{ \bigwedge_{u \in U} s \in S_u(\bb a)}$.

\subsection{}
The first step towards estimating the inner sums in \eqref{EXC:eq:39} and \eqref{EXC:eq:39x} is to obtain a uniform upper bound on the probabilities $\PP\bra{ \bigwedge_{u \in U} s \in S_u(\bb a)}$. This will allow us to eliminate a small proportion of summands for which more accurate estimates are difficult to obtain. In order to accommodate later applications in Section \ref{section:HM} we prove a statement that is slightly more general than what is currently required. 

\begin{lemma}\label{EXC:lem:kill-bad-U}
	Fix an integer $M \geq 1$ and let $n$, $\bb a$ and $\e$ be as introduced above. For $1 \leq k \leq M$, let $s_k$ and $u_k$ be integers with $1 \leq s_k \leq \binom{n+1}{2}$ and $\e n \leq u_k \leq n$. Assume further that $u_k \neq u_l$ for all $k,l$ with $k \neq l$ and $1 \leq l,k \leq M$. Then
	\begin{equation}\label{EXC:eq:40a}
		\PP\bra{ \bigwedge_{k=1}^M s_k \in S_{u_k}(\bb a)} = O_M\bra{ \bra{ \frac{ {\log n} }{\e n}}^M},
	\end{equation}
	where the implicit constant depends only on $M$.
\end{lemma}
\begin{proof}
	We may assume that $n$ is sufficiently large in terms of $M$, since otherwise the statement it trivially true. Put $C = 2M+6$. It follows from Corollary \ref{EXC:cor:Hoeffding} (applied with $\tau = \sqrt{(M+1)\log n}$) that the probability that there exist $u,v$ with $1 \leq u < v \leq n+1$ such that $v - u \geq C \log n$ and $\sum_{i=u}^{v-1} \bb a_i \leq n$ is less than $2n^{-M}$.
	
	For $j$ with $1 \leq j \leq n+1$ let $\cA_{j}$ denote the event that \begin{inparaenum}[(i)]
\item $s_k \in S_{u_k}(\bb a)$ for all $k$ with $1 \leq k \leq M$ and $u_k \geq j$ and
\item $\sum_{i=u}^{v-1} \bb a_i > n$ for all $u,v$ with $j \leq u < v \leq n+1$ and $v-u \geq C \log n$. 
\end{inparaenum}
Note that whether or not $\cA_{j}$ holds is fully determined by $\bb a_{i}$ for $i$ with $j \leq i \leq n$. Moreover,  $\cA_j$ implies $\cA_{j+1}$ for each $j$ with $1 \leq j \leq n$ and $\cA_{n+1}$ is vacuously true. Also,
	\begin{equation}\label{EXC:eq:90}	
	\PP\bra{ \bigwedge_{k=1}^M s_k \in S_{u_k}(\bb a)} \leq \PP(\cA_1) + 2n^{-M}.
	\end{equation}

	Consider a process where $\bb a_j$ are selected in the order of decreasing $j$. For each $j$ with $1 \leq j \leq n$, conditional on $\bb a_{j+1}, \dots, \bb a_{n}$, the value of $\bb a_j$ is uniformly distributed among the $j$ elements of $[n]$ which have not previously been selected. We will estimate $\PP(\cA_j)$ in terms of $\PP(\cA_{j+1})$. For each $j$ we have the trivial bound
	\begin{equation}\label{EXC:eq:91x}	
	\PP(\cA_j)/\PP(\cA_{j+1}) \leq 1.
	\end{equation}
	Suppose next that $j = u_k$ for some $1 \leq k \leq M$, so in particular $j \geq  \e n$. If $\cA_{j+1}$ holds then there at most $C \log n$ values of $v$ such that $s_k-n \leq \sum_{i=u_{k+1}}^v < s_k$, meaning that there are at most $C \log n$ possible values of $\bb a_j$ such that $s_k \in S_{u_k}(\bb a)$. It follows that
	\begin{equation}\label{EXC:eq:91}	
	\PP(\cA_j)/\PP(\cA_{j+1}) = \PP(\cA_j | \cA_{j+1}) \leq \frac{C \log n}{\e n}.
	\end{equation}
	Applying \eqref{EXC:eq:91} for all $j \in \{u_k\}_{k=1}^M$ and \eqref{EXC:eq:91x} for all remaining $j$ leads to
	\[ \PP(\cA_1) \leq  \bra{\frac{ {C \log n} }{\e n}}^M,\]
	which together with \eqref{EXC:eq:90} implies \eqref{EXC:eq:40a}.
\end{proof}

\subsection{}\label{SSec:P-wedge-s-in-S}
The probabilities $\PP\bra{ \bigwedge_{u \in U} s \in S_u(\bb a)}$ appearing in \eqref{EXC:eq:39} and \eqref{EXC:eq:39x} are easier to estimate under certain genericity assumptions on the index set $U$. 
We will say that a sequence of indices $(u_k)_{k=1}^M$ with $1 \leq u_k \leq n$ for all $k$ is \emph{well-separated} if $u_k$ are separated gaps of length $\sqrt{n}$, i.e,
\begin{equation}\label{EXC:cond:U-reg}
 \abs{u_k - u_l} \geq \sqrt{n} \text{ for all } k,l \text{ with } 1 \leq k < l \leq M.
\end{equation}
Accordingly, we say that a set $U = \{u_i\}_{i=1}^M$ with $\abs{U} = M$ is well-separated if \eqref{EXC:cond:U-reg} holds. The choice of $\sqrt{n}$ is somewhat arbitrary, but convenient in the applications.

For well-separated index sets we can obtain a tight estimate for the probabilities under investigation, from which Proposition \ref{EXP:prop:main} will ultimately follow.
\begin{proposition}\label{EXC:prop:P-wedge-s-in-S}
	Fix an integer $M \geq 1$ and let $n$, $\bb a$ and $\e$ be as introduced above. For $1 \leq k \leq M$ let $s_k = \sigma_k \binom{n+1}{2}$ be integers with $\e \leq \sigma_k \leq 1-\e$ and let $u_k$ with $\e n \leq u_k \leq (1-\sigma_k - \e) n$. Suppose further that the sequence $(u_k)_{k=1}^M$ is well-separated (c.f.\ \eqref{EXC:cond:U-reg}). Then 
\begin{equation}
	\PP\bra{ \bigwedge_{k=1}^M { s_k \in S_{u_k}(\bb a) } } = \bra{\ffrac{2}{n}}^M \bra{1+ O_M\bra{n^{-\frac{1}{40}}}},
\end{equation}
where the implicit error term depends only on $M$ and the choice of $\e$.
\end{proposition}

\begin{proof}[Proof of Proposition \ref{EXC:prop:P-s-in-S} assuming Proposition \ref{EXC:prop:P-wedge-s-in-S}]
	Let $N$ be a large even integer. From \eqref{EXC:eq:39} we have
	\begin{equation}\label{eq:EXC:349}
		\PP(s \not\in S'(\bb a)) \leq \sum_{M \leq N} (-1)^M \sum_{\abs{U} = M} \PP\bra{ \bigwedge_{u \in U} s \in S_u(\bb a) },
	\end{equation}
	where the inner sum runs over all sets of integers $U \subset [\e n, (1-\sigma-\e) n]$ with $\abs{U} = M$.
	Let $n'$ denote the number of integers $u$ with $\e n \leq u \leq (1-\sigma-\e)n$. Then
	\( n' = (1-\sigma-2\e)n + O(1) \) and the number of summands in the inner sum in \eqref{eq:EXC:349} is 
	\[\binom{n'}{M} = \frac{((1-\sigma)n)^M}{M!} +O_M(\e n^M).\]
The number of summands for which $U$ is not well-separated (i.e., \eqref{EXC:cond:U-reg} fails) is $O_M(n^{M-\frac{1}{2}})$, and by Lemma \ref{EXC:lem:kill-bad-U} each of these summands has size bounded by $\displaystyle O_M\braBig{\brabig{ \frac{ \log n}{\e n}}^M}$. Hence, the total contribution to \eqref{eq:EXC:349} from summands corresponding to sets $U$ that are not well-separated is bounded by 
\[
O_M\bra{n^{-\frac{1}{2}} \bra{ \frac{ \log n}{\e n}}^M} = O_M(n^{-\frac{49}{100}}).
\]
By Proposition \ref{EXC:prop:P-wedge-s-in-S}, all remaining summands are of the form
	\[\bra{\frac{2}{n}}^M \bra{1+O_M(n^{-\frac{1}{40}})}.\]
	Thus, the inner sum in \eqref{eq:EXC:349} can be estimated as
\begin{align*}
	\sum_{\abs{U} = M} \PP\bra{ \bigwedge_{u \in U} s \in S_u(\bb a) } &= 
	\bra{ \frac{(1-\sigma)^M}{M!}n^M + O_{M}\bra{\e n^M} } \cdot \bra{\frac{2}{n}}^M \bra{ 1 + O_M(n^{-\frac{49}{100}})}
	\\&= \frac{ (2-2\sigma)^M }{M!} + O_M(\e).
\end{align*}
	This leads to the upper bound (valid as long as $N \geq 4$)
	
\begin{align*}
\PP(s \not \in 	(\bb a)) &\leq \sum_{M=0}^N (-1)^M \frac{ (2-2\sigma )^M }{M!} + O_N(\e)
\leq e^{-2+2\sigma} + \frac{2^N}{N!} + O_N(\e).
\end{align*}		
 Letting $N \to \infty$ slowly with $n$ we conclude that $\PP(s \in S'(\bb a)) \leq  e^{-2+2\sigma} + o(1)$, where the error term is uniform with respect to the choice of $s$ and depends at most on our choice of $\e$.
	 
	 Running the same argument with $N$ odd we obtain the reversed bound $\PP(s \in S'(\bb a)) \geq  e^{-2+2\sigma} - o(1)$, and hence $\PP(s \not\in S'(\bb a)) = e^{-2+2\sigma}+o(1)$.
\end{proof}

\subsection{}
Our strategy of proof of Proposition \ref{EXC:prop:P-wedge-s-in-S} relies on showing that the events $ s_k \in S_{u_k}(\bb a) $ are approximately independent and then estimating their probabilities separately.  In order to accomplish this, we consider a somewhat convoluted procedure for selecting $\bb a$, which we now sketch out.

Let $m$ be a moderately large even integer ($m = 2 \floor{n^{{3}/{19}}}$ will turn out to be a convenient choice). For each $k$ with $1 \leq k \leq M$ we first select a (random) set $\bb A_k \subset [n]$ from which $\bb a_i$ will be selected for $u_k \leq i < u_k+m$. This is useful because when $\bb A_k$ are substantially larger than $m$ then the distribution of the sum $\sum_{i=u_k}^{u_k+m-1} \bb a_{i}$ (conditional on $\bb A_k$, for most choices of $\bb A_k$) can  be accurately approximated. In fact, this distribution is close to the distribution of the sum of $m$ independent random variables distributed uniformly on $[n]$. 

Suppose now that for some $k$ we have already selected the set $\bb A_k$, as well as $\bb a_i$ for $i \geq u_k+m$. Then the probability that $s_k \in S_{u_k}(\bb a)$ is very close to the probability that the sum of $m$ independent random variables distributed uniformly on $[n]$ belongs to the set  $s_k-S_{u_k+m}(\bb a)$. The latter probability is relatively easy to estimate and, under mild additional assumptions, is essentially independent of what set $\bb A_k$ and what entries $\bb a_i$ for $i \geq u_k+m$ were chosen. This yields the approximate independence mentioned above. 

We now put this plan into action.

\subsection{}

For a set $A \subset [n]$ and an integer $k$ with $1 \leq k \leq \abs{A}$, we let $\f_{A,k}\colon \ZZ \to [0,1]$ denote the probability mass function of the sum $\bb x = \bb x_1 + \bb x_2 + \dots \bb x_k$ of $k$ random variables $\bb x_1, \bb x_2, \dots, \bb x_k$ sampled from $A$ without replacement, that is,
\[\f_{A,k}(x) = \PP(\bb x = x) = \PP\bra{\sum_{i=1}^k \bb x_i = x}, \qquad (x \in \ZZ).\]
It will also be convenient to extend this definition to $k > \abs{A}$, setting $\f_{A,k}(x) = 0$ for all $x \in \ZZ$.
Likewise, we let $\f_{\uni,k} \colon \ZZ \to [0,1]$ denote the probability mass function of the sum $\bb y = \bb y _1 + \bb y _2 + \dots + \bb y _k$ of $k$ independent random variables $\bb y_i$ distributed uniformly on $[n]$.  For instance, $\f_{A,1} = \frac{1}{\abs{A}} 1_A$ and  $\f_{\uni,1} = \frac{1}{n} 1_{[n]}$, where $1_X$ denotes the indicator function of a set $X$.

Recall that for a bounded function $f \colon \ZZ \to \RR$, its supremum norm is given by $\norm{f}_{\infty} = \sup_{x \in \ZZ} \abs{f(x)}$.
For two functions $f,g \colon \ZZ \to \RR$ with finite supports, their convolution $f \ast g$ is given by \[f \ast g (x) = \sum_{y \in \ZZ} f(x-y) g(y),\qquad (x \in \ZZ).\] 

We will say that a non-empty set $A \subset [n]$ of size $\alpha n$ is \emph{uniform}, or $\sU(A)$ for short, if for each $l \geq 1$ it holds that
\begin{equation}\label{eq:EXC:def-of-U(A)}
		\norm{ \f_{\bb A,2l} - \f_{\uni,2l} }_{\infty} \leq 15 n^{-3/2} l^3 \alpha^{-2} \log n.
\end{equation} 
While we formally include all $l \geq 1$ in the definition, \eqref{eq:EXC:def-of-U(A)} is trivially true for sufficiently large $l$ since the expression on the right hand side is larger than $1$. The upper bound in \eqref{eq:EXC:def-of-U(A)} was chosen so that we have very good control on the probability that a random subset of $[n]$ is not uniform, given by the following result.
 
The following lemma shows that a random set of a given size is uniform with overwhelming probability.

\begin{lemma}\label{EXC:lem:f-A,2=uni}
\label{EXC:lem:f-A,m=uni}
	Let $n \geq 1$ be an integer, let $\alpha \in (0,1]$ be such that $\alpha n \in \NN$, and let $\bb A$ be a set chosen uniformly at random from all subsets of $[n]$ with cardinality $\alpha n$. Then 
\begin{equation}\label{eq:924:1}
	 \PP\left( \neg \sU(\bb A) \right) \leq \exp(-4\log^2 n + O(\log n)).
\end{equation}
\end{lemma}
\begin{proof}
	For each integer $l \geq 1$ we will show that 
	\begin{equation}\label{eq:924:2}
	\PP\left(	\norm{ \f_{\bb A,2l} - \f_{\uni,2l} }_{\infty} \geq 15 n^{-3/2} l^3 \alpha^{-2} \log n  \right) \leq \exp(-4\log^2 n + O(\log n)).
	\end{equation}
	Once this is accomplished, in order to prove the estimate \eqref{eq:924:1} it is enough to apply the union bound over all $l$ with $1 \leq l \leq n$ and to recall that the uniformity condition \eqref{eq:EXC:def-of-U(A)} is trivially true for $l > n$.
	
	We begin with the case $l = 1$. For each integer $x$ with $2 \leq x \leq 2n$ we will show that 
	\begin{equation}\label{eq:924:3}
	\PP\left(
		\abs{ 1_{\bb A} * 1_{\bb A}(x) - \alpha^2 1_{[n]} * 1_{[n]}(x) } \geq 7 \sqrt{n} \log n
	\right) \leq \exp(-4\log^2 n + O(1)).		
	\end{equation}
	Note that for $x$ with $x < 2$ or $x > 2n$ we have
	\[  1_{\bb A} * 1_{\bb A}(x) = \alpha^2 1_{[n]} * 1_{[n]}(x) = 0, \]
	so applying to \eqref{eq:924:3} the union bound over all $x$ with $2 \leq x \leq 2n$ we obtain 
		\begin{equation}\label{eq:924:4}
	\PP\left(	\norm{ \f_{\bb A,2} - \f_{\uni,2} }_{\infty} \geq 7 n^{-3/2} \alpha^{-2} \log n  \right) \leq \exp(-4\log^2 n + O(\log n)),
	\end{equation}
	which is marginally stronger than \eqref{eq:924:2} with $l = 1$. Since both $1_{\bb A} * 1_{\bb A}$ and $1_{[n]} * 1_{[n]}$ are symmetric with respect to $x = n+1$, it will suffice to consider $x$ with $2 \leq x \leq n+1$. We will also assume, for notational convenience, that $x$ is odd; the argument in the case where $x$ is even is fully analogous.
	
	Pick an odd integer $x$ with $2 \leq x \leq n+1$. Define $\bb B = \bb A \cap [1,x/2)$ and $\bb C = \bb A \cap (x/2,x-1]$. Then $1_{[n]} * 1_{[n]}(x) = x-1$ and $1_{\bb A} * 1_{\bb A}(x) = 2\abs{\bb B \cap (x- \bb C)}$. Note that $1_\bb A, 1_\bb B$ and $1_\bb C$ can be construed as sequences of random variables, sampled without replacement from a multiset containing $\alpha n$ copies of $1$ and $(1-\alpha) n$ copies of $0$. Hence, it follows from Hoeffding's inequality that 
	\begin{align}\label{eq:924:5}
		\PP\left( \abs{ \abs{\bb B} - \alpha (x-1)/2 } \geq \sqrt{x-1} \log n \right) &\leq 2 \exp(-4 \log^2 n), \\
		\PP\left( \abs{ \abs{\bb C} - \alpha (x-1)/2 } \geq \sqrt{x-1} \log n \right) \label{eq:924:6}
		&\leq 2 \exp(-4 \log^2 n).
	\end{align}
	
For any set of integers $C \subset (x/2,x-1]$ with $\abs{C} \leq \alpha n$, conditional on $\bb C = C$, $1_{\bb B}$ can again be construed as a sequence of random variables sampled without replacement, this time from a multiset containing $\alpha n - \abs{C}$ copies of $1$ and $(1-\alpha) n - (x-1)/2 + \abs{C}$ copies of $0$. Applying Hoeffding's inequality again, we conclude that
\begin{align*}
	\PP \left( \abs{
			\abs{\bb B \cap (x- \bb C)} - \frac{2\abs{\bb B}\abs{\bb C}}{x-1}
		} 
		\geq \sqrt{x-1} \log n	
	\middle| \bb C = C \right) &\leq 2 \exp\left( \frac{ -2(x-1) \log^2 n}{ \abs{C} }\right) \\	
	&\leq 2 \exp\left(-4 \log^2 n \right).
\end{align*}
Averaging over all the possible choices of $C$, we conclude that 
\begin{equation}\label{eq:924:8}
	\PP \left( \abs{
			\abs{\bb B \cap (x- \bb C)} - \frac{2\abs{\bb B}\abs{\bb C}}{x-1}
		} 
		\geq \sqrt{x-1} \log n \right) \leq 2 \exp\left(-4 \log^2 n \right).
\end{equation}

By the union bound applied to \eqref{eq:924:5}, \eqref{eq:924:6} and \eqref{eq:924:8}, with probability at least $1-6\exp(-4\log^2 n )$ we have the following chain of inequalities:
	
	\begin{align*}
		\abs{ 1_{\bb A} * 1_{\bb A}(x) - \alpha^2 1_{[n]} * 1_{[n]}(x) } 
		 &= \abs{2 \abs{\bb B \cap (x- \bb C)} - \alpha^2 (x-1) } 
		\\& \leq  \abs{ \frac{4 \abs{\bb B} \abs{\bb C}}{x-1} - \alpha^2 (x-1) } + 2\sqrt{x-1} \log n 
		\\& \leq 4 \log^2 n + (4 \alpha +2)\sqrt{x-1} \log n 
		\leq 7 \sqrt{n}\log n,
	\end{align*}
	where in the last transition we assume, as we may, that $n$ is sufficiently large. This finishes the proof of \eqref{eq:924:3} 	
	
	Consider now $l > 1$. We may assume without loss of generality that $l\leq \alpha n/2$, since the uniformity condition \eqref{eq:924:1} automatically holds is $l > \alpha n/2$. Let $l'$ be the integer such that $l' \floor{\alpha n/l} + (l-l') \ceil{\alpha n/l} = \alpha n$, and let $\bb A = \bb A_1 \cup \dots \cup \bb A_l$ be a partition of $\bb A$ into cells $\bb A_j$ ($1 \leq j \leq l$) with cardinalities $\abs{\bb A_j} = \floor{\alpha n/l}$ if $1 \leq j \leq l'$ and $\abs{\bb A_j} = \ceil{\alpha n/l}$ if $l' < j \leq l$, chosen uniformly at random. Note that each of the cells $\bb A_j$ ($1 \leq j \leq l$) contains at least two elements. If $A \subset [n]$ is a set with $\abs{A} = \alpha n$ then, conditional on $\bb A = A$, choosing a pair of elements from each of the cells $\bb A_j$ ($1 \leq j \leq l$) uniformly at random yields a $2l$-tuple with the same distribution as sampling $2l$ entries of $A$ without replacement, and hence 
	\[ \f_{A,2l} = \EE \left( \f_{\bb A_1,2} \ast \dots \ast \f_{\bb A_l,2} \ \middle| \ \bb A = A\right).\]
	As a consequence,
	\begin{equation}\label{eq:925:1}
		\f_{\bb A,2l} = \EE \left( \f_{\bb A_1,2} \ast \dots \ast \f_{\bb A_l,2} \ \middle| \ \bb A \right).
	\end{equation}	 

	It is a standard fact that if $f$ is a probability distribution then for any bounded sequence $g \colon \ZZ \to \RR$ one has $\norm{ f \ast g}_{\infty} \leq \norm{g}_\infty$. Hence, by a standard telescoping argument we find that 
	\begin{equation}\label{eq:925:2}
\norm{ \f_{\bb A_1,2} \ast \f_{\bb A_2,2} \ast \dots \ast \f_{\bb A_{l},2} - \f_{\uni,2l} }_\infty \leq \sum_{j=1}^{l} \norm{\f_{\bb  A_j,2} - \f_{\uni,2}}_\infty.
	\end{equation}	
For a set $B \subset [n]$ with cardinality $\beta n$, let $\cB(B)$ denote the statement 
	\[\norm{\f_{B,2}-\f_{\uni,2}}_\infty \geq 7 n^{-3/2} \beta^{-2} \log n.\]
	Combining \eqref{eq:925:2} and \eqref{eq:925:1},  can now estimate	
	\begin{align*}
	 \norm{  \f_{\bb A,2l} - \f_{\uni,2l} }_\infty 
	 &\leq \sum_{j=1}^{l} \EE \left( \norm{  \f_{\bb A_j,2} - \f_{\uni,2} }_\infty  \ \middle| \ \bb A \right) 
	 \\ & \leq 7 l' n^{1/2} \floor{\alpha n/l}^{-2} \log n +
	7 (l-l') n^{1/2} \ceil{\alpha n/l}^{-2} \log n +
	\sum_{j=1}^{l}\PP\left( \cB(\bb A_j )  | \bb A \right)
	\\ & \leq 14 n^{-3/2} \alpha^{-2} l^3 \log n +
	\sum_{j=1}^{l}\PP\left( \cB(\bb A_j )  | \bb A \right).
	\end{align*}

	For each integer $j$ with $1 \leq j \leq l$, the set $\bb A_j$ is uniformly distributed among all subsets of $[n]$ with cardinality $\floor{\alpha n/l}$ or $\ceil{\alpha n/l}$ (depending on the value of $j$), so it follows from \eqref{eq:924:4} that
	\begin{equation*}\label{eq:925:3}
\PP(\cB(\bb A_j)) \leq \exp(-4 \log^2 n + O(\log n)).
	\end{equation*}	
	By the Markov inequality, 
	\begin{equation}\label{eq:925:4}
 \PP \left( \PP( \cB(\bb A_j | \bb A ) \geq n^{-3/2} \alpha^{-2} l^2 \log n \right) \leq \exp( - 4 \log^2 n +O(\log n) ).
	\end{equation}	
	Inserting \eqref{eq:925:4} into the previously obtained estimate on $\norm{  \f_{\bb A,2l} - \f_{\uni,2l} }_\infty$ and applying the union bound over all $j$ with  $1 \leq j \leq l$ we conclude that
\[	\PP\left( 
	\norm{  \f_{\bb A,2l} - \f_{\uni,2l} }_\infty \geq 15 l n^{-3/2} l^3 \alpha^{-2} \log n 	
	\right)\leq \exp( - 4 \log^2 n +O(\log n) ),
\]
which is precisely \eqref{eq:924:2}.
\end{proof}

\subsection{}

We record some basic facts about sums of independent random variables uniformly distributed on $[n]$. The key insight, made precise in Lemma \ref{lem:EXC:P-x-in-R}, is that if $k \geq 2$ is an even integer then $\f_{\uni,k}$ is smooth enough to remove irregularities on the scale of roughly $n \sqrt{k}$. We begin with a simple lemma.

\begin{lemma}\label{EXC:lem:f_uni-basic}
	Let $n \geq 1$ be an integer, and let $k \geq 2$ be an even integer. Put $x_0 = k(n+1)/2$. Then
	\begin{enumerate}
	\item\label{EXC:item-90-A} $\f_{\uni,k}(x)$ is increasing for $x \leq x_0$, decreasing for $x \geq x_0$, and symmetric with respect to $x = x_0$;
	\item\label{EXC:item-90-B} $\f_{\uni,k}(x) \leq C/\sqrt{k}n$ for all integers $x$ where $C > 0$ is an absolute constant.
	\item\label{EXC:item-90-C} for each $t > 0$ it holds that $\displaystyle \sum_{\abs{y} \geq t} \f_{\uni,k}(x_0+y) \leq 2\exp(-2t^2/kn^2)$.
	\end{enumerate}
\end{lemma}
\begin{proof}
For \eqref{EXC:item-90-A}, proceed by induction on $k$. The case $k = 2$ is clear by direct inspection. It remains to notice that if the claim is true for $k$ and $l$ then it is also true for $k+l$ since $\f_{\uni,k+l} = \f_{\uni,k} \ast \f_{\uni,l}$ (see e.g.{} \cite[Thm.{} 1.6]{book-unimodular}).

To deal with \eqref{EXC:item-90-B}, we introduce the exponential sum 
\[\phi(t) := \frac{1}{n} \sum_{j=1}^n e( j t) = \frac{e(t)}{n} \frac{e(nt)-1}{e(t)-1}\]
 where $e(t) = e^{2 \pi i t}$, so that
 \[\f_{\uni,k}(x) = \int_{0}^1 e(-xt) \phi(t)^k dt \leq \int_0^k \abs{\phi(t)}^k dt.\]
We have the trivial bound $\abs{\phi(t)} \leq 1$ for each $t \in [0,1]$, as well as the identity $\abs{\phi(t)} = \abs{ \frac{\sin n t\pi}{n \sin t \pi}}$. We also recall a standard estimate $2 x \leq \sin \pi x \leq \pi x$ for $0 \leq x \leq 1/2$, which implies that $\abs{\phi(t)} \leq (2nt)^{-1}$. It follows that 
\begin{equation}\label{eq:392:1}
	\int_{1/2n}^{1-1/2n} \abs{\phi(t)}^k dt \leq 2 \int_{1/2n}^\infty (2nt)^{-k} dt = \frac{1}{(k-1)n}
\end{equation}

For $0 < t < 1/n$ we use slightly more careful estimates. It follows from elementary analysis that there exist constants $c_1, c_2 > 0$ such that \[\pi x e^{-c_1 x^2} \leq \sin \pi x \leq \pi x e^{-c_2 x^2} \text { for } 0 \leq x \leq 1/2\]
(one can take $c_2 = \ffrac{\pi^2}{6}$ and $c_1 = 4 \log(\pi/2)$). We may (increasing $C$ if necessary) assume without loss of generality that $n$ is large enough that $c_2 n^2 > 2c_1$, and hence 
\[ \abs{ \phi(t) } \leq \frac{\exp(-c_2 n^2t^2)}{\exp(-c_1 t^2)} \exp(-c_2 n^2 t^2/2). \] As a consequence,
\[
\int_{0}^{1/2n} \abs{ \phi(t) }^k dt \leq \int_{0}^\infty \exp(-c_2 k n^2 t^2/2) dt = \bra{\frac{ \pi }{ 2 c_2 k n^2}}^{1/2}.\]
The same estimate holds by essentially the same reasoning for the remaining integral $\int_{1-1/2n}^{1} \abs{ \phi(t) }^k dt$. The claim \eqref{EXC:item-90-B} now follows by combining the partial estimates obtained above.

Finally, \eqref{EXC:item-90-C} is a direct application of the Hoeffding inequality.
\end{proof}

\newcommand{\vart}{t}
\newcommand{\varT}{T}
\begin{lemma}\label{lem:EXC:P-x-in-R}
	Let $n,k,\vart,T \geq 1$ be integers with $k$ even and $T \geq 2t$, let $\rho \in [0,1]$ and let $R \subset \ZZ$. Put $x_0 := k(n+1)/2$ and 
\[
	\eta := \max_{y \in [x_0-T,x_0+T-t)} \abs{ \frac{\abs{R \cap [y,y+t)}}{t} - \rho }.
\]
	Assume further that $\eta < \rho$. Then
	\begin{equation}\label{eq:728:1}
	\abs{ \sum_{x \in R} \f_{\uni,k}(x) - \rho } \leq \eta + 2 \exp\left(-\frac{T^2}{kn^2}\right) + O\left(\frac{t \rho}{\sqrt{k}n}\right).
	\end{equation}
\end{lemma}
\begin{proof}
	For $i \in \ZZ$ put $x_i = x_0 + i t$. 
	If follows from Lemma \ref{EXC:lem:f_uni-basic}\eqref{EXC:item-90-A} that
	\[
		t\sum_{i = -\infty}^{\infty} \f_{\uni,k}(x_i) - t\f_{\uni,k}(x_0) \leq \sum_{x = -\infty}^{\infty}  \f_{\uni,k}(x_i) = 1 \leq t\sum_{i = -\infty}^{\infty} \f_{\uni,k}(x_i) + t f_{\uni,k}(x_0).
	\]
	In particular, it follows from Lemma \ref{EXC:lem:f_uni-basic}\eqref{EXC:item-90-B} that
	\begin{equation}\label{eq:728:2}
	\sum_{i = -\infty}^{\infty} \f_{\uni,k}(x_i)  = 1/t + O(1/\sqrt{k}n).
	\end{equation}
	Let $i_{\max} := \floor{T/t}$. Since $i_{\max}t \geq T/2$, it follows from Lemma \ref{EXC:lem:f_uni-basic}\eqref{EXC:item-90-C} that
	\begin{equation}\label{eq:728:3}
	\sum_{x \geq x_{i_{\max}} } \f_{\uni,k}(x) + \sum_{x \leq x_{-i_{\max}} } \f_{\uni,k}(x) \leq 2 \exp\left( - \frac{T^2}{kn^2} \right).
	\end{equation}
		Using piecewise monotonicity of $\f_{\uni,k}$ again, we observe that
	\begin{align*}
			\sum_{x \in R} \f_{\uni,k}(x) \leq& \sum_{i=-i_{\max}+1}^{0} \abs{ R \cap [x_{i-1}, x_{i})} \f_{\uni,k}(x_{i}) 
			\\ & + \sum_{i=0}^{i_{\max}-1} \abs{ R \cap [x_i, x_{i+1})} \f_{\uni,k}(x_{i}) + 2 \exp\left( - \frac{T^2}{kn^2} \right).
	\end{align*}
Using \eqref{eq:728:2}, we now conclude that
	\begin{align*}
			\sum_{x \in R} \f_{\uni,k}(x) &\leq 
			\sum_{i=-i_{\max}}^{i_{\max}} t (\rho+\eta) \f_{\uni,k}(x_{i}) + t (\rho+\eta) f(x_0)  + 2 \exp\left( - \frac{T^2}{kn^2} \right)
			\\ & \leq \rho  + \eta + O(t\rho/\sqrt{k}n)  + 2 \exp\left( - \frac{T^2}{kn^2} \right).
	\end{align*}
The estimate in the opposite direction,
	\begin{align*}
			\rho - \sum_{x \in R} \f_{\uni,k}(x) \leq  \eta + O(t\rho/\sqrt{k}n)  + 2 \exp\left( - \frac{T^2}{kn^2} \right).
	\end{align*}
 follows by a fully analogous reasoning, the only significant difference being that instead of \eqref{eq:728:2} we use the estimate
\begin{equation*}
	t \sum_{\abs{i} > i_{\max} } \f_{\uni,k}(x_i) \leq 2 \exp\left( - \frac{T^2}{kn^2} \right). \qedhere
	\end{equation*}
\end{proof}

\subsection{}
Another component of the proof of Proposition \ref{EXC:prop:P-wedge-s-in-S} is the observation that sums $\sum_{i=u}^{v-1} \bb a_i$ are close to their average values, provided that $v-u$ is larger than $\log n$. Moreover, this behaviour is preserved if $\bb a_i$ are chosen from a set $A \subset [n]$ that is sufficiently large and obeys certain mild additional assumptions. To make this precise, we introduce the following piece of notation.

Let $I \subset [n]$ be a set and let $b = (b_i)_{i \in I}$ be a sequence  indexed by $I$, taking values in $[n]$. We will say that $b$ is \emph{regular}, or $\cR(b)$ for short, if 
	\begin{equation} 
		\abs{ \sum_{i \in J} b_i - \abs{J} \ffrac{(n+1)}{2} } \leq \sqrt{\abs{J}} n \log n.
		\label{EXC:eq:R} 
	\end{equation}
	for each interval $J \subset I$.
It is a standard consequence of the Hoeffding inequality (cf.{} Corollary \ref{EXC:cor:Hoeffding}) that for any interval $I \subset [n]$,
\[ 
\PP\bra{ \neg \cR( (\bb a)_{i \in I} ) } \leq 2\exp(-2\log^2 n).
\]
Similarly, for a set $A \subset [n]$ we will say that $A$ is \emph{regular}, or $\cR(A)$ for short, if for an ordering $\bb b = (\bb b_j)_{j=1}^{\abs{A}}$ of $A$ chosen uniformly at random we have
	\begin{equation}\label{eq:EXC:def-R} 
\PP( \neg \cR(\bb b) ) \leq \exp(-\log^2 n).
	\end{equation} 
These definitions are set up so that a random set is regular with high probability, as shown by the following lemma.

\begin{lemma}\label{EXC:lem:P-R-uniform}
	Let $n \geq $ be an integer, as introduced above, let $\alpha \in (0,1]$ with $\alpha n \in \NN$, and let $\bb A$ be a subset of $[n]$ of size $\alpha n$, chosen uniformly at random. Then
\[	\PP( \neg \cR(\bb A) ) \leq \exp(-\log^2 n + O(\log n)). \]
\end{lemma}
\begin{proof}
	Let $\bb b = (\bb b_i)_{i=1}^{\alpha n}$ be an ordering of $\bb A$ chosen uniformly at random. Note that $\bb b$ has the same distribution as a sequence of $\alpha n$ entries sampled without replacement from $[n]$. Hence, by the Hoeffding inequality, for any $u,v$ with $1 \leq u < v \leq \alpha n +1$ we have
	\[
		\PP\bra{ \abs{ \sum_{i=1}^{v-1} \bb b_i - (v-u)(n+1)/2} \geq \sqrt{(v-u)} n \log n } \leq 2 \exp(-2\log^2 n).
	\]
By the union bound we now obtain 
\[ 
\PP( \neg \cR(\bb b)) \leq \exp(-2\log^2 n + O(\log n)).
\]
 It follows by the Markov bound that
\begin{align*}
	\PP( \neg \cR(\bb A)) 
	&= \PP\braBig{ \PP\bra{\neg \cR(\bb b) | \bb A} > \exp(-\log^2 n) }
	\\ &\leq \exp(\log^2 n) \EE\brabig{ \PP\bra{\neg \cR(\bb b) | \bb A}} 
	\leq \exp(-\log^2 n + O(\log n)). \qedhere
\end{align*}	
\end{proof} 

Later in the argument, we will also encounter concatenations of regular sequences. While that concatenation of regular sequences is not guaranteed to be regular, a slightly weaker statement is true. If $(I_{j})_{j=1}^M$ is a sequence of intervals, $J \subset I = \bigcup_{j=1}^{M} I_j$ is an interval and $(b_i)_{i \in I}$ is a sequence such that for each $j$, the restriction $b|_{I_j} = (b_i)_{i \in I_{j}}$ is regular then    
	\begin{equation} 
		\abs{ \sum_{j \in J}b_i - \abs{J} \ffrac{(n+1)}{2} } \leq \sqrt{M \abs{J}} n \log n.
		\label{EXC:eq:R2} 
	\end{equation}
This follows directly from the application of the definition of regularity to $J \cap I_j$ for each $1 \leq j \leq M$, and the inequality between the quadratic and arithmetic means.
\begin{lemma}\label{lem:EXC:regular-interval}
	Let $n \geq 1$ be an integer, as introduced above, and let $a = (a_i)_{i=1}^{k}$ be a sequence taking values in $[n]$. Assume that there exists a partition $[k] = \bigcup_{j=1}^M I_j$ such that $\cR(a|_{ I_j})$ holds for each $1 \leq j \leq M$. Let $K \subset \NN$ be a interval and assume that $\max K \leq \ffrac{k(n+1)}{2} - \sqrt{Mk} n\log n$  and $\abs{K} \geq M n \log^2 n$. Then	
	\begin{equation} 
\abs{K \cap S_1(b)} = \frac{2 \abs{K} }{n+1} + O\bra{M\sqrt{\abs{K}}\log n/\sqrt{n}}.		\label{EXC:eq:J/K} 
	\end{equation}
\end{lemma}
\begin{proof}
	Let $J$ denote the set of $v$ with $1 \leq v \leq k+1$ such that $\sum_{i=1}^{v-1} a_i \in K$. Note that $\abs{J} = \abs{K \cap S_1(b)}$ and $J$ is an interval. On one hand, by regularity we have
	\begin{equation}\label{eq:610:3}
		\abs{ \sum_{i \in J} a_i - \frac{(n+1)\abs{J}}{2} } \leq \sqrt{M \abs{J}} n \log n.
	\end{equation}
	As a consequence, through a series of elementary manipulations we obtain
	\begin{equation}\label{eq:610:1}
		\abs{ \sqrt{\abs{J}} - \sqrt{ \frac{2}{n+1}\sum_{i \in J} a_i + M \log^2 n} } \leq \sqrt{M} \log n.
	\end{equation}
	
	On the other hand, we note that the gaps between consecutive elements of $S_1(a)$ are at most $n$. Moreover,
	\[  \min S_1(a) \leq n < \min K +n \] 
and (by another application of regularity)
	\[ \max S_1(a) = \sum_{i=1}^k a_i \geq \frac{k(n+1)}{2} - \sqrt{Mk} n\log n \geq \max K.\]
As a consequence, $\min (S_1(a) \cap K) = \min K + O(n)$ and $\max (S_1(a) \cap K) = \max K + O(n)$, and hence 
	\begin{equation}\label{eq:610:2}
		\sum_{i \in J} a_i = \max (S_1(a) \cap K) - \min (S_1(a) \cap K) + O(n) = \abs{K} + O(n)
	\end{equation}
Inserting \eqref{eq:610:2} in \eqref{eq:610:1} leads to
	\begin{equation}\label{eq:610:4}
		\sqrt{\abs{J}} = \sqrt{ \frac{2\abs{K}}{n+1} } + O(\sqrt{M} \log n).
	\end{equation}
Taking squares in \eqref{eq:610:4} yields \eqref{EXC:eq:J/K}.
\end{proof}

\subsection{}
We now have gathered all the ingredients needed for the proof of Proposition \ref{EXC:prop:P-wedge-s-in-S}. Recall that for a sequence $a = (a_i)_{i\in I}$ and an index set $J \subset I$ the restriction of $a$ to $J$ is denoted by $a|_J = (a_i)_{i \in J}$.

\begin{proof}[Proof of Proposition \ref{EXC:prop:P-wedge-s-in-S}]
	Put $m := 2 \floor{n^{\frac{3}{19}}}$ and  $\a = \e/M$. Assume without loss of generality that $u_1 > u_2 > \dots > u_M$, and put $v_k := u_k+m$ for all $1 \leq k \leq M$. We assume for the sake of brevity that $\alpha n$ is an integer (the argument without this assumption is fully analogous) and that $n$ is sufficiently large in terms of $M$.
	
	For $k$ with $1 \leq k \leq M$ we define the sets
	\[I_k = \left[u_k,v_k\right) \cup \left((k-1)(\alpha n - m),k(\alpha n - m)\right]\] and we additionally put $I_0 = [n] \setminus\bigcup_{k=1}^M I_k$. The sets $I_k$ ($0 \leq k \leq M$) are pairwise disjoint subsets of $[n]$; this is the case since $v_{k} < u_{k+1}$ ($1 \leq k < M$) thanks to the well-separation assumption, and $v_M \leq n-\e n + m \leq n$, and $u_1 \geq \e n > M(\alpha n -m)$. Moreover, for each $1 \leq k \leq M$ the set $I_k$ has cardinality $\alpha n$ and consists of the union of $[u_k,v_k)$ and a remainder part contained in $[1,\e n)$. 
	
	We will, informally, think of $\bb a$ as being selected in two stages: first we select which entries appear in $\bb a|_{I_k}$ for $0 \leq k \leq M$, and second we choose the order in which they appear. Formally, for $0 \leq k \leq M$ we define the random sets $\bb A_k = \left\{ \bb a_i \ : \ i \in I_k \right\}$. By Lemma \ref{EXC:lem:f-A,m=uni} and Lemma \ref{EXC:lem:P-R-uniform} respectively, each of the conditions $\sU(\bb A_k)$ ($0 \leq k \leq M$) and $\cR(\bb A_k)$ ($0 \leq k \leq M$) holds with probability $1-n^{-\omega(1)}$. Hence, it will suffice to show that for any sets $A_k$ ($0 \leq k \leq M$) satisfying the conditions $\sU(A_k)$ and $\cR(A_k)$ mentioned above it holds that
\begin{equation}\label{EXC:eq:49}
	\PP\bra{ \bigwedge_{k=1}^M {s_k \in S_{u_k}(\bb a)} \ \middle| \ \bb A_k = A_k  \text{ for } 0 \leq k \leq M } = \bra{\frac{2}{n}}^M \bra{1+ O_M(n^{-\frac{1}{40}})}.
\end{equation}
For each $0 \leq k \leq M$, directly from the definition of $\cR(A_k)$ in \eqref{eq:EXC:def-R} we have
\begin{equation}\label{EXC:eq:941-a}
\PP\Big(\neg \cR( \bb a|_{I_k} ) \ \Big| \ \bb A_k = A_k\Big) = n^{-\omega(1)},
\end{equation}
and the same estimate remains valid if some of the entries of $\bb a$ outside of $I_k$ are specified. This will allow us to freely replace each event $s_k \in S_{u_k}(\bb a)$ in \eqref{EXC:eq:49} with the event $s_k \in S_{u_k}(\bb a) \wedge \cR(\bb a|_{I_k})$ at the cost of introducing a negligible error term of the order $n^{-\omega(1)}$.

We will next prove the following estimate: Let $1 \leq k \leq M$ and suppose that for all $l$ with $0 \leq l < k$, orderings $(a_i)_{i\in I_l}$ of $A_l$ are specified and $\cR\bra{(a_i)_{i\in I_l}}$ holds; then
	\begin{equation}\label{EXC:eq:50}
		\PP \left( s_k \in S_{u_k}(\bb a)  
		\ \middle| \
		 \begin{matrix}
		\bb A_k = A_k \text{ for } 0 \leq k \leq M\text{,} \\
		 \bb a|_{I_l} = a|_{I_l} \text{ for } 0 \leq l < k
		 \end{matrix} 
		 \right)   = \frac{2+O_M(n^{-\frac{1}{40}})}{n}.
	\end{equation}
Once this has been proved, the proposition follows by a standard inductive argument. More precisely, for $L = 0,1,\dots,M$ we prove that
\begin{equation}\label{EXC:eq:49A}
	\PP\bra{ \bigwedge_{k=1}^L {s_k \in S_{u_k}(\bb a)} \ \middle| \ \bb A_k = A_k  \text{ for } 0 \leq k \leq M } = \bra{\frac{2}{n}}^L \bra{1+ O_M(n^{-\frac{1}{40}})}.
\end{equation}
The case $L = 0$ is trivial, while the case $L = M$ is \eqref{EXC:eq:49}. The inductive step amounts to expressing the left hand side of \eqref{EXC:eq:49A} as the weighted sum of probabilities that take the form \eqref{EXC:eq:50} with $k = L$, but without the guarantee that the conditions $\cR\bra{(a_i)_{i\in I_l}}$ hold for $l < L$. Where possible, we estimate the summands using \eqref{EXC:eq:50}, and we bound the remainder using \eqref{EXC:eq:941-a}. This, together with the inductive assumption, yields the formula \eqref{EXC:eq:49A}.

We proceed to the proof of \eqref{EXC:eq:50}. Fix $k$ ($1 \leq k \leq M$) and the values of $a_i$ for $i \in \bigcup_{l=0}^{k-1} I_l$. 
Under the conditions in \eqref{EXC:eq:50}, $s_k \in S_{u_k}(\bb a)$ if and only if $s_k \in \sum_{i=u_k}^{v_k-1} \bb a_i + S_{v_k}(a)$, which in turn is equivalent to $\sum_{i=u_k}^{v_k-1} \bb a_i \in s_k -  S_{v_k}(a)$. (Note that $\sum_{i=u_k}^{v_k-1} \bb a_i \leq mn < \e \binom{n+1}{2} \leq s_k$.) Put $R = s_k -  S_{v_k}(a)$ and let $\bb x$ be the sum of $m$ independent variables uniformly distributed on $[n]$. 
By the definition of $\sU(A_k)$ in \eqref{eq:EXC:def-of-U(A)} we have
	\begin{equation}\label{EXC:eq:51}
\PP\left( \sum_{i=u_k}^{v_k-1} \bb a_i \in R \ \middle| \ \bb A_k = A_k \right) = \PP( \bb x \in R ) + O\bra{n^{-\frac{3}{2}} m^{3} \e^{-2} \log n}.
	\end{equation}
	
Put $t := \floor{m^{\frac{1}{3}}n}$ and note that $M n \log^2 n \leq t \leq s_k/2$. By Lemma \ref{lem:EXC:regular-interval}, for each interval $K \subset \left[0, s_k \right]$ of length $t$, we have
\[\abs{K \cap R } = \frac{2 t}{n+1}+ O(\eta t),
\text{ where } \eta = \eta(t) = \frac{M\log n}{\sqrt{t n}}.\]
(Here we use the fact that $s_k \leq (n-u_k)(n+1)/2 - \e \binom{n+1}{2}$.) 
Applying Lemma \ref{lem:EXC:P-x-in-R} and noting that $\eta < 2/(n+1)$, we conclude that
	\begin{equation}\label{EXC:eq:52}
	\PP( \bb x \in R) = \frac{2}{n+1} + O(\eta) + O\bra{\exp\bra{-\frac{s_k^2}{4mn^2}} } + O\bra{\frac{t}{\sqrt{m} n^2}}.
	\end{equation}
The exponential error term is much smaller than the other two, and hence can be disregarded. The two remaining error terms both have size $O_M(m^{-\frac{1}{6}} n^{-1} \log n)$. Combining \eqref{EXC:eq:51} and \eqref{EXC:eq:52} we conclude that the probability on the right hand side of \eqref{EXC:eq:50} is equal to
	\begin{align*}\label{EXC:eq:55}
	\frac{2}{n+1} 
	+ O\bra{n^{-\frac{3}{2}} m^{3} \e^{-2} \log n} 
	+ O_M\bra{n^{-1} m^{-\frac{1}{6}} \log n}
	= \frac{2}{n} + O_{M}(n^{-\frac{39}{38}+o(1)}).
	\end{align*}
This implies  \eqref{EXC:eq:50}, even with a slightly better error term.	
\end{proof} 
\section{Higher moments}\label{section:HM}

\subsection{} 
In this section, we obtain an asymptotic formula for the second moment $\EE\abs{S(\bb a)}^2$, where like in the previous section $\bb a$ is an element of $\Sym([n])$ chosen uniformly at random and $n$ is a large integer. As a consequence, we prove the concentration around the mean for $\abs{S(\bb a)}$. By a standard application of the second moment method, 
\begin{equation}\label{eq:HM:001}
	\PP\bra{ \abs{\abs{S(\bb a)}^{\phantom{|}\!\!\!} - \EE\abs{S( \bb a)}} > \delta n^2 } \leq \frac{\EE{\abs{S(\bb a)}^2} - (\EE\abs{S(\bb a)})^2}{\delta^2 n^4}
\end{equation}
for any $\delta > 0$. 
Hence, Theorem \ref{thm:INT:B} will follow directly from the following result.

\begin{proposition}\label{HM:prop:E(S(a)^p)}
	Let $n$ and $\bb a$ be as introduced above. Then 
\begin{equation}\label{eq:HM:002}	
	\EE \abs{S(\bb a)}^2 = (c^2 + o(1)) n^{4},
\end{equation}
where $c = \ffrac{(1+e^{-2})}{4}$. 
\end{proposition}

Because $\abs{S(\bb a)}/n^2$ is bounded for $a \in \Sym([n])$, the concentration around the mean, as stated in Theorem \ref{thm:INT:B}, implies the asymptotic formula for the higher moments $\EE \abs{S(\bb a)}^p$ for all $p\geq 1$, namely 
\[ \EE \abs{S(\bb a)}^p = (c^p +o(1))n^{2p}.\]
This formula also follows directly from a slight adaptation of the argument we give here.

We will argue along the similar lines as in Section \ref{section:EXC}. The only missing ingredient we need in order to compute the higher moments of $\abs{S(\bb a)}$ is an analogue of Lemma \ref{EXC:lem:kill-bad-U} which is applicable to sequences of indices $u_1,u_2,\dots,u_M$ that include repeated entries. Proof of this result is less trivial than it might appear at first, and occupies the most of this section. The following example hints at the source of complications.

\begin{example}\label{exple:EXC:001}
	Pick $u_1,u_2,u_3$ and $s_1,s_2,s_3$ with $\frac{1}{3}n \leq u_1,u_2,u_3 \leq \frac{2}{3}n$ and $\frac{1}{10} \binom{n+1}{2} \leq s_1,s_2,s_3 \leq \frac{2}{10} \binom{n+1}{2}$ satisfying the following constraints: 
	
	\begin{align*}
	u_1 = u_3 < u_2, \quad s_3 = s_1 + s_2, \quad
	\EE \braBig{\sum_{i=u_1}^{u_2-1} \bb a_i } = (u_2-u_1)(n+1)/2 = s_1.
\end{align*} 
	In analogy with Lemma \ref{EXC:lem:kill-bad-U}, one might expect that $\PP\bra{\bigwedge_{k=1}^3 s_k \in S_{u_k}(\bb a)} \leq n^{-3+o(1)}$. However, in order for $\bigwedge_{k=1}^3 s_k \in S_{u_k}(\bb a)$ to hold, it is sufficient that 
	\[ s_1 = \sum_{i=u_1}^{u_2-1} \bb a_i \ \text{ and  } \ s_2 \in S_{u_2}(\bb a).\]
However, one can check (since we only use this example as a motivation, we omit proof) that 
	\[ \PP\braBig{ s_1 = \sum_{i=u_1}^{u_2-1} \bb a_i } = n^{-\frac{3}{2} + o(1)} \ \text{ and  } \ \PP\bra{ s_2 \in S_{u_2}(\bb a)} = n^{-1+o(1)}.\]
Additionally, the two events are approximately independent, and hence
	\[ \PP\braBig{ s_1 = \sum_{i=u_1}^{u_2-1} \bb a_i \text{ and  }  s_2 \in S_{u_2}(\bb a)} = n^{-\frac{5}{2}+o(1)},\]
contrary to the expectation based on Lemma \ref{EXC:lem:kill-bad-U}.
\end{example}

\subsection{} 

As Example \ref{exple:EXC:001} suggests, the direct generalisation of Lemma \ref{EXC:lem:kill-bad-U} is not possible; instead we prove an averaged version.  As before, we will only apply the following lemma in the case when $s_k$ take at most two distinct values, but this constraint does not significantly simplify the reasoning. Throughout this section, we let $\e = \e(n) > 0$ denote a positive real, subject to the same constraints as introduced in Section \ref{section:EXC}.

\begin{lemma}\label{EXC:lem:kill-bad-U-2}
	Fix an integer $M \geq 1$. Let $n,\bb a$ and $\e$ be as introduced above, and for $1 \leq k \leq M$ let $s_k = \sigma_k \binom{n+1}{2}$ be integers with $\e \leq \sigma_k \leq 1-\e$. Let $\cU$ denote the  set of all increasing sequences $u= (u_k)_{k=1}^M$ such that $\e n \leq u_k \leq (1-\sigma_k - \e)n$ for all $k$ with $1 \leq k \leq M$, and there exist $k,l$ with $1 \leq k < l \leq M$ such that $u_k = u_l$. Then
	\begin{equation}\label{EXC:eq:71}
		\sum_{u \in \cU} \PP\bra{ \bigwedge_{k=1}^M s_k \in S_{u_k}(a)} = O_M\bra{ \ffrac{\bra{ \frac{\log n}{\e}}^M}{n}},
	\end{equation}
	where the implicit constant depends only on $M$.
\end{lemma}

Once we prove the above lemma, we will have all the tools necessary to compute the second moment of $\abs{S(\bb a)}$. The argument is very similar to the one we used to compute $\EE \abs{S(\bb a)}$. In places where the arguments are virtually identical, we give only the outline, and refer the Reader to the relevant parts of Section \ref{section:EXC} for details. 

\begin{proof}[Proof of Proposition \ref{HM:prop:E(S(a)^p)} assuming Lemma \ref{EXC:lem:kill-bad-U-2}]
We will show that for any integers $s_1,s_2$ with $s_1 \neq s_2$ and $s_j = \sigma_j \binom{n+1}{2}$, $\e \leq \sigma_j \leq 1-\e$ it holds that
\begin{equation}\label{EXC:eq:61}
	\PP\brabig{  s_1,s_2 \not \in S'(\bb a)} = e^{-2+2\sigma_1} e^{-2+2\sigma_2} + o(1).
\end{equation}
where $S'(\bb a)$ is defined as in \eqref{eq:EXC:002} in Section \ref{section:EXC} and the error term is uniform with respect to the choice of $s_1,s_2$. Once this is established, by Proposition \ref{EXC:prop:P-s-in-S} it also follows that
\[\PP\brabig{ s_1,s_2 \in S'(\bb a)} = (1-e^{-2+2\sigma_1}) (1-e^{-2+2\sigma_2}) + o(1).\]
and hence by a Riemann integral approximation argument similar to the one in the proof of Proposition \ref{EXP:prop:main}, 

\begin{align*}
	\EE\abs{S'(\bb a)}^2 &= \sum_{s_1,s_2} \PP\brabig{ s_1,s_2 \in S'(\bb a)} + o(n^4)
	\\ &= n^4 \int_{\e}^{1-\e}\int_{\e}^{1-\e} (1-e^{-2+2\sigma_1}) (1-e^{-2+2\sigma_2}) d \sigma_1 d\sigma_2 + o(n^4) 
	\\&= (c^2 + o(1))n^4,
\end{align*} 
where the sums run over all $s_1,s_2$ as specified above, and $c = \ffrac{1 + e^{-2}}{4}$. Finally, we note that almost surely $\abs{S(\bb a)} = \abs{S'(\bb a)} + o(n^2)$, as we have already shown in subsection \ref{SSec:proof-of-thm-B}. Hence, it remains to prove \eqref{EXC:eq:61}.

Using the inclusion--exclusion formula like in subsection \ref{SSec:inclusion-exclusion}, we may rewrite the probability in \eqref{EXC:eq:61} as
\begin{equation}\label{EXC:eq:62}
	\sum_{M_1=0}^\infty (-1)^{M_1} \sum_{\abs{U_1} = M_1} \sum_{M_2=0}^\infty (-1)^{M_2} \sum_{\abs{U_2} = M_2} \PP\bra{ \bigwedge_{j=1}^2 \bigwedge_{k=1}^{M_j} { s_j \in S_{u_{j,k}}(\bb a) }},
\end{equation}
where the inner sums are taken over all choices of $U_{j} = \{u_{j,k}\}_{k=1}^{M_j} \subset [n]$ such that $\e n < u_{j,k} < (1-\e-\sigma_j) n$ ($j \in \{1,2\}$). 
Let $N$ be a large even integer, and put $N_1 = N$ and 
\[	
	N_2(M_1) = 
	\begin{cases}
		N & \text{ if } M_1\equiv 0 \pmod{2} \\
		N+1 & \text{ if } M_1 \equiv 1 \pmod{2}.		
	\end{cases}
\] 
Then the sum in \eqref{EXC:eq:62} is bounded from above by the truncated sum
\begin{equation}\label{EXC:eq:62a}
	\sum_{M_1=0}^{N_1} (-1)^{M_1}
	\sum_{M_2=0}^{N_2(M_1)} (-1)^{M_2} 
	\sum_{\substack{ \abs{U_1} = M_1 \\ \abs{U_2} = M_2}}
	\PP\bra{ \bigwedge_{j=1}^2 \bigwedge_{k=1}^{M_j} { s_j \in S_{u_{j,k}}(\bb a) }}.
\end{equation}

With the same definitions, for odd values of $N$ the expression in \eqref{EXC:eq:62a} gives a lower bound for \eqref{EXC:eq:62}. Hence, like in subsection \ref{SSec:P-wedge-s-in-S}, to find asymptotics for the sum \eqref{EXC:eq:62}, it will suffice to find asymptotics for each of the innermost sums in \eqref{EXC:eq:62a}.

Fix $M_1,M_2$ and consider one such sum. Put $M = M_1+M_2$. Using Lemma \ref{EXC:lem:kill-bad-U-2}, we may disregard the contribution $n^{-1+o(1)}$ coming from $U_1, U_2$ such that $U_1 \cap U_2 \neq \emptyset$, and using Lemma \ref{EXC:lem:kill-bad-U} we may disregard the contribution $n^{-\frac{1}{2}+o(1)}$ coming from $U_1,U_2$ with $U_1 \cap U_2 = \emptyset$ such that $U_1 \cup U_2$ is not well-separated (as defined in \eqref{EXC:cond:U-reg}). For remaining $U_1,U_2$, by Proposition \ref{EXC:prop:P-wedge-s-in-S} we have
\begin{equation}\label{EXC:eq:63}
	\sum_{\substack{ \abs{U_1} = M_2 \\ \abs{U_1} = M_2 }}
	\PP\bra{ \bigwedge_{j=1}^2 \bigwedge_{k=1}^{M_j} \bra{ s_j \in S_{u_{j,k}}(\bb a) }}
 = \bra{ \frac{2}{n} }^{M}\bra{1 + O_M\bra{n^{-\frac{1}{40}}}}.
\end{equation}
The number of choices of $U_1$ and $U_2$ for given values of $M_1,M_2$ is 
\[ \frac{(1-\sigma_1)^{M_1}(1-\sigma_2)^{M_2}}{M_1! M_2!}n^{M}+O_M\bra{\e n^M }.\]
Thus the inner sum in \eqref{EXC:eq:62a} is, up to an error of size $O_M\bra{\e}$, equal to
\begin{equation}\label{EXC:eq:64}
\bra{ \frac{2}{n} }^{M} \frac{(1-\sigma_1)^{M_1}(1-\sigma_2)^{M_2}}{M_1!M_2!}n^{M} 
= \frac{(2-2\sigma_1)^{M_1}(2-2\sigma_2)^{M_2}}{M_1!M_2!}.
\end{equation}
Thus, for any large even integer $N$ the sum in \eqref{EXC:eq:62a} is bounded from above by
\begin{equation}\label{EXC:eq:65}
	\sum_{M_1=0}^{\infty} \sum_{M_2=0}^{\infty} (-1)^{M_1+M_2}
	\frac{(2-2\sigma_1)^{M_1} (2-2\sigma_2)^{M_2} }{M_1! M_2!}
	+ \frac{2^N}{N!} + O_N(\e).
\end{equation}
The sum in \eqref{EXC:eq:65} is simply the Taylor expansion of $e^{-2+2\sigma_1}e^{-2+2\sigma_2}$, and letting $N \to \infty$ slowly with $n$ we may combine the two error terms into an error term $o(1)$ (uniformly in $s_1,s_2$). It follows that
\begin{equation}\label{EXC:eq:61x}
	\PP\brabig{  s_1,s_2 \not \in S'(\bb a)} \leq e^{-2+2\sigma_1} e^{-2+2\sigma_2} + o(1).
\end{equation}
Repeating the same argument for $N$ odd, we obtain the corresponding inequality in reverse direction, which combined with \eqref{EXC:eq:61x} yields \eqref{EXC:eq:61} and finishes the proof.
\end{proof}

\subsection{}
 We devote the remainder of this section to the proof of Lemma \ref{EXC:lem:kill-bad-U-2}. Fix $M \geq 1$, and let $s_k$ $(1 \leq k \leq M$) be as in the formulation of the lemma. From this point, we allow all implicit constants to depend on $M$. 

In order to record the possible linear dependencies between the sums $\sum_{i=u}^{v-1} \bb a_i$ ($1 \leq u < v \leq n+1$) and the target values $s_k$ ($1 \leq k \leq M$) we introduce a structure which we call a ``type graph'' (it is a graph which encodes the type of dependency).

Recall that a $\ZZ$-labelled directed graph $G$ consists of the following data:
\begin{enumerate}
	\item the vertex set $V = V_G$;
	\item the set of edges $E = E_G \subset V^2$;
	\item the edge labels $r = r_G \in \ZZ^{E}$.
\end{enumerate}
We will be particularly interested in graphs whose vertex set is a subset of $\NN$. In this case, we define the total edge length of $G$ as  
\begin{equation}\label{eq:def-LG}
L_G = \sum_{(k,l) \in E} \abs{k-l}.
\end{equation}

\begin{definition}
	A \emph{type graph} $G$ is a $\ZZ$-labelled directed graph such that
{\begin{enumerate}
	\item the vertex $V = V_G$ is the initial segment $[N]$ for some integer $N = N_G$;
	\item the edges are directed so that if $(k,l) \in E$ then $k < l$;
	\item the edge labels $r = r_G \in \ZZ^{E}$ take the form $r_{e} = \sum_{k=1}^M q_{e,k} s_k$ for some coefficients $q_{e,k} \in \ZZ$ ($e \in E_G,\ 1\leq k \leq M$) obeying the constraint
\begin{equation}
 \sum_{k=1}^M \abs{q_{e,k}} \leq 2^{N^2-L} \ \text{ for each } \ e \in E.
 \end{equation}
\end{enumerate}
}\end{definition}
The significance of the last condition will become clear in the course of the argument. When there is no risk of confusion, we omit the subscript $G$ in $N_G$, $E_G$, $r_G$ and $L_G$, and write simply $N,E,r$ and $L$.

For a type graph $G$, a permutation $a \in \Sym([n])$ is a permutation of $[n]$ and a sequence $w = (w_k)_{k =1}^N$ taking values in $[n+1]$, we introduce the event $\cA_G(a,w)$, specified by
\begin{equation}\label{EXC:eq:78}
	\cA_G(a,w) \iff \text{ $w$ is increasing and } \sum_{i=w_k}^{w_l-1} a_i = r_{{k,l}}, \text{ for all } \bra{k,l} \in E.
\end{equation}

Two type graphs $G$ and $H$ are said to be \emph{equivalent} if they define the same events and additionally have the connected components. More precisely, $G \sim H$ if $N_G = N_H = N$, the events $\cA_G(a,w)$ and $\cA_H(a,w)$ are equivalent for each $a \in \Sym([n])$ and $w \in [n+1]^N$, and the graphs $([N],E_G)$ and $([N],E_H)$ (or, strictly speaking, the induced undirected graphs) have the same connected components. This clearly defines an equivalence relation.
We say that a type graph $G$ is \emph{satisfiable} if it is possible to satisfy the corresponding event $\cA_G$, i.e., if there exist $a \in \Sym([n])$ and $w \in [n+1]^N$  such that $\cA_G(a,w)$ holds. The property of being satisfiable is preserved under equivalence. Type graphs which are not satisfiable will not play a significant role.  
We will call a type graph $G$ \emph{minimal} if it is satisfiable and has minimal total edge length $L_G$ within its equivalence class. Each satisfiable equivalence class has at least one minimal element.

\begin{example}\label{exple:EXC:002}
	Let $(u_k)_{k=1}^{3}$ and $(s_k)_{k=1}^3$ be as in Example \ref{exple:EXC:001}. Consider the type graph $G$ with $N = 3$, $E = \{ (1,2), (2,3), (1,3) \}$ and labels $r_{1,2} = s_1$, $r_{2,3} = s_2$, $r_{1,3} = s_3$. Then for $w = (u_1,u_2,v)$ the event $\cA(a, w)$ is equivalent to the system of equations
\begin{equation}\label{eq:EXC:010}
		\sum_{i=u_1}^{u_2-1} a_i = s_1,\qquad \sum_{i=u_2}^{v-1} a_i = s_2,\qquad \sum_{i=u_1}^{v-1} a_i = s_3. 
\end{equation}
	Hence, $\bigwedge_{k=1}^3 s_k \in S_{u_k}(\bb a)$ holds if and only if $\cA_G(\bb a, (u_1,u_2,v))$ holds for some $v$ with $u_2 < v \leq n+1$.
	Note that any of three equations in \eqref{eq:EXC:010} can be eliminated, whence $G$ is equivalent to any of the type graphs obtained by removing one of the edges. If $G'$ is described in the same manner as $G$ but with $s'_3 \neq s_1' + s_2'$ then $G'$ is not satisfiable.
\end{example}

\begin{remark}
	Recall that the definition of equivalence of type graphs includes the requirement that if $G \sim H$ then $G$ and $H$ have the same connected components. Although we believe that in most cases of interest this follows from the requirement that the conditions $\cA_G$ and $\cA_H$ are equivalent, there are some pathological examples when this is not the case. For instance, if $G$ is a type graph with $E_G = \{(1,3),(2,4)\}$ and $r_{1,3} = 3,\ r_{2,4} = 4$, and $H$ is the type graph obtained by adding to $G$ the edge $(2,3)$ with $r_{2,3} =  1$ then one can check that the conditions $\cA_G$ and $\cA_H$ are equivalent, while the connected components of $G$ and $H$ are different. Problems also arise when $N = n+1$, since the edge $(1,n+1)$ with label $r_{1,n+1} = \binom{n+1}{2}$ may be freely added without altering the corresponding condition. 
	
In order to avoid the complications described above, we simply include the equality of connected components in the definition of equivalence. Alternatively, one could pose the definition of equivalence of type graphs $G$ and $H$ requiring not only that the conditions $\cA_G$ and $\cA_H$ are equivalent, but also that the systems of equations in \eqref{EXC:eq:78} defining $\cA_G(a,w)$ and $\cA_H(a,w)$ for a given choice of $w$ are equivalent as, say, systems of polynomial equations in variables $(a_i)_{i=1}^n$. 
\end{remark}

\subsection{}
We record several basic properties of minimal type graphs. 
\begin{lemma}\label{lem:HM:001}
	With the above notation, if $G$ is a minimal type graph then the underlying graph $([N_G],E_G)$ is a union of pairwise disjoint paths.
\end{lemma}
\begin{proof}
We will show that each vertex is the initial point and the end point of at most one edge. Once this is accomplished, it is easy to see that $([N],E)$ takes the required form. We will only deal with end points since the statement for initial points if fully analogous.

Suppose for the sake of contradiction that $E$ contains the edges $(k,m)$ and $(l,m)$ for some $k,l$ with $1 \leq k < l < m \leq N$. Consider the type graph obtained from $G$ by removing the edge $(k,m)$ and adding the edge $(k,l)$ with label $r_{k,l} = r_{k,m} - r_{l,m}$. (If $G$ already contains the edge $(k,l)$ then necessarily $r_{k,l} = r_{k,m} - r_{l,m}$ since $G$ is satisfiable; in this case we simply remove the edge $(k,m)$.) Note that $L_H \leq L_G-1$, so $r_{k,l}$ defined above is a feasible label (that is, it satisfies \eqref{eq:def-LG}). The events defined by $G$ and $H$ are equivalent, and $G$ and $H$ have the same connected components. Hence, $G \sim H$, which contradicts the minimality of $G$ and finishes the argument.
\end{proof}

The following lemma elucidates the connection between the events $\cA_G$ for a type graph $G$ and the events $s_k \in S_{u_k}(\bb a)$ appearing in \eqref{EXC:eq:71}, already hinted at in Example \ref{exple:EXC:002}. The length of a path is the number of edges it contains.

\begin{lemma}\label{lem:HM:002}
	With the above notation, if $(u_k)_{k=1}^M \in \cU$, $a \in \Sym([n])$ and $s_k \in S_{u_k}(a)$ for each $1 \leq k \leq M$, then there exists a minimal type graph $G$ with $N_G < 2M$ and an increasing sequence $(w_l)_{l=1}^{N_G}$ such that $\cA_G(a,w)$ holds, $w_1 = u_1$, $\{w_l\}_{l=1}^{N_G} \supset \{u_k\}_{k=1}^M$, and $G$ contains a path of length $\geq 2$.
\end{lemma}
\begin{proof}
	Since for each $k$ with $1 \leq k \leq M$ we have $s_k \in S_{u_k}(a)$, it follows that there exists $v_k$ with $u_k < v_k \leq n+1$ such that $\sum_{i=u_k}^{v_k-1} a_i = s_{k}$. Let $w = \{w_l\}_{l=1}^{N}$ be the increasing sequence obtained by arranging $u_k$ and $v_k$ ($1 \leq k \leq M)$ in an increasing order, that is $\{w_l\}_{l=1}^{N} = \{u_k\}_{k=1}^M \cup \{v_k\}_{k=1}^M$. Note that $N < 2M$ since $u \in \cU$.
	We define $E$ and $r$ by letting $\bra{l,m} \in E$ if and only if $(w_l,w_m) = (u_k,v_k)$ for some $1 \leq k \leq M$, and putting $r_{{l,m}} = s_k$ for such $k$. In principle, one edge $(l,m)$ may correspond to several values of $k$, but the fact that $s_k = \sum_{i=u_k}^{v_k-1} a_i = \sum_{i=w_l}^{w_m-1} a_i$ implies that the definition of $r_{l,m}$ is well posed. 
	
	It follows directly from the above construction that $\cA_G(a,w)$ holds for the type graph $G$ specified by $N,E$ and $r$. Moreover, $G$ has $M$ edges and strictly fewer than $2M$ vertices, so at least one of its connected component has size $\geq 3$. It remains to replace $G$ with a minimal type graph in the same equivalence class and recall Lemma \ref{lem:HM:001}.	
\end{proof}

\subsection{} We are are now ready to finish the argument.

\begin{proof}[Proof of Lemma \ref{EXC:lem:kill-bad-U-2}]
	Since the number of type graphs $G$ with strictly fewer that $2M$ vertices is $O_M(1)$, it follows from Lemma \ref{lem:HM:002} that that in order to prove Lemma \ref{EXC:lem:kill-bad-U-2}, it will suffice to check that for each minimal type graph $G = ([N],E,r)$ that has fewer than $2M$ vertices and contains a path of length at least $2$ we have
	\begin{equation}\label{EXC:eq:75}
		\PP\bra{ \exists w \in [n+1]^N \text{ s.t. } \cA_G(\bb a, w) \text{ and } w_1 \geq \e n } 
		= O_M\bra{ \ffrac{\bra{ \frac{\log n}{\e}}^M}{n}}.
	\end{equation}

Recall that by Lemma \ref{lem:HM:001}, $G$ is a union of disjoint paths. Let $F \subset [N]$ denote the set of the initial points of these paths, i.e., an integer $l$ with $1 \leq l \leq N$ belongs to $F$ if and only if  $(k,l) \not \in E$ for all $k$ with $1 \leq k \leq N$. Note that $1 \in F$. By the union bound, in order to prove \eqref{EXC:eq:75} it will suffice to show that for each $v = (v_k)_{k \in F} \in [n+1]^F$ with $v_1 \geq \e n$ we have
	\begin{equation}\label{EXC:eq:75A}
		\PP\bra{ \exists w \in [n+1]^N \text{ s.t. } \cA_G(\bb a, w) \text{ and } w|_{F} = v} 
		= O_M\bra{ \ffrac{\bra{ \frac{\log n}{\e}}^M}{n^{\abs{F}+1}}}.
	\end{equation}

Intuitively speaking, we will now select the entries $\bb a_i$ in the order of increasing $i$, starting with $i_0 := \floor{ \e n}$. At ``time'' $t$ with $0 \leq t < n$, we select $\bb a_{i(t)}$ where $t_{\max} := n - i_0$ and
\[
i(t) := 
\begin{cases}
i_0 + t &\text{ if } 0 \leq t \leq t_{\max};\\
t-t_{\max}  &\text{ if } t_{\max} < t \leq n.
\end{cases}
\]
The entries selected at times $t > t_{\max}$ do not play an important role in the argument.
To make this idea precise, for each $t$ with $0 \leq t \leq n$, let $\cF_t$ denote the $\sigma$-algebra generated by $\bb a_{i(t')}$ with $0 \leq t' < t$. In particular $\cF_{0}$ is trivial and $\bb a$ is $\cF_n$-measurable. 

We next introduce a sequence of random variables $\bb w = (\bb w_k)_{k=1}^N$ with the property that, if $\cA_G(\bb a, w)$ holds for some $w \in [n+1]^N$ with $w|_{F} = v$, then $\bb w = w$. We define $\bb w_k$ for $k = 1,2,\dots,N$ inductively. If $k \in F$, we set $\bb w_k = v_k$. Otherwise, there exists exactly one $l$ with $1 \leq l < k$ such that $(l,k) \in E$, and we put
	\begin{equation}\label{EXC:eq:75B}
	\bb w_k = 
		\min \left\{ v \ : 1 \leq v \leq n+1, \ \sum_{j = \bb w_l}^{v-1} \bb a_j \geq r_{l, k} \right\},
	\end{equation}
	where the minimum of an empty set is defined to be $\infty$. In order to prove \eqref{EXC:eq:75A}, it will suffice to show that 
	\begin{equation}\label{EXC:eq:75Z}
		\PP\bra{ \cA_G(\bb a, \bb w) } = O_M\bra{ \ffrac{\bra{ \frac{\log n}{\e}}^M}{n^{\abs{F}+1}}}.
	\end{equation}
Although in general there is no guarantee that $\bb w$ is increasing, for each edge $(l,k) \in E$ we have $\bb w_l < \bb w_k$. It follows from a simple inductive argument that the event $\bb w_k \leq i(t)$ is $\cF_{t}$-measurable for each $k,t$ with $1 \leq  k \leq N$ and $0 \leq t \leq t_{\max}+1$. 

In order to avoid the need to write out increasingly convoluted formulas, for an edge $e = (k,l) \in E$ and integer $t$ with $0 \leq t \leq t_{\max}+1$ we introduce the shorthand
\begin{equation}\label{EXC:eq:75C}
	\bb \Sigma(e,t) = 
	\begin{cases}
	r_{k,l} - \sum_{j= \bb w_l}^{i(t-1)}  \bb a_{j} &\text{ if } \bb w_l < i(t),\\
	\infty &\text{ if } \bb w_l \geq i(t).
	\end{cases}		
	\end{equation}
Hence, $\bb \Sigma(e,t) > 0$ if and only $\bb w_k > i(t)$, and $\sum_{i=\bb w_l}^{\bb w_k -1} \bb a_i = r_{k,l}$ if and only if it is possible to choose $t$ so that $\bb \Sigma(e,t) = 0$. As a function of $t$, $\bb \Sigma(e,t)$ is decreasing for each $e \in E$. Additionally, $\bb \Sigma(e,t)$ is $\cF_t$-measurable for each $e \in E$ and $0 \leq t \leq t_{\max}+1$.

Let $\bb X$ be the set consisting of those $t$ with $0 \leq t \leq t_{\max} + 1$ for which these exists $e \in E$ with $\bb \Sigma(e,t) = 0$. Hence, if $\cA_G(\bb a, \bb w)$ holds, then $\bb X = \{ \bb w_{k} - i_0 \ : \ k \in [N] \setminus F \}$ and in particular $\abs{\bb X} = \abs{E}$; we will estimate the probability of the latter event. For each $t \in \bb X$ there exists $e \in E$ with 
\[ \bb \Sigma(e,t-1) = \bb a_{i(t-1)}+\bb \Sigma(e,t) = \bb a_{i(t-1)} \in [n].\]
This motivates us to introduce the set $\bb Y$, consisting of those $t$ with $0 \leq t \leq t_{\max} + 1$ for which these exists $e \in E$ with $\bb \Sigma(e,t) \in [n]$. In particular, $\bb X \subset \bb Y + 1$. For each $t$, the events $t \in \bb X$ and $t \in \bb Y$ are $\cF_t$-measurable. Intuitively, $\bb Y$ is the set of those ``times'' when we have a chance to add another element to $\bb X$ in the next step. Hence, one can hope to obtain estimates of the size of $\bb X$ in terms of the size of $\bb Y$.

For each $t$ with $0 \leq t \leq t_{\max}$, conditional on $\cF_t$, $\bb a_{i(t)}$ is chosen uniformly at random from a set of $\geq \e n$ possible values, and hence, using the union bound over all possible choices of the edge $e \in E$ which could satisfy $\bb \Sigma(e,t+1) = 0$ we obtain
\begin{equation}\label{EXC:eq:900}
		\PP(t+1 \in \bb X \ | \ \cF_t) \leq 2M 1_{\bb Y}(t)/\e n.
\end{equation}

We next show that \eqref{EXC:eq:900} implies that for each $h \geq \abs{E}$ we have the estimate 
\begin{equation}\label{EXC:eq:901}
	\PP\brabig{ \abs{\bb X} = \abs{E} \wedge \abs{\bb Y} \leq h } = O_{M} \braBig{ \brabig{\frac{h}{\e n}}^{\abs{E}} }. 
\end{equation}
In order to prove \eqref{EXC:eq:901}, define for each set $J \subset \NN$ the random set $\bb Y_J \subset \bb Y$ obtained by selecting from $\bb Y$ the elements at positions in $J$. More precisely, if $\bb Y = \{t_i \, : \, 1 \leq i \leq \abs{\bb Y}\}$ with $t_1 < t_2 < \dots$, then $\bb Y_J := \{t_i \ : \ 1 \leq i \leq \abs{\bb Y},\ i \in J\}$. By the union bound, for any integer $x \geq 0$ we have
\begin{equation}\label{EXC:eq:902}
	\PP \braBig{ \abs{\bb X} = x \wedge \abs{\bb Y} \leq h }
	\leq \sum_{\substack{J \subset [h],\ \abs{J} = x}} \PP\braBig{ \bb X = (\bb Y_J + 1) \wedge \ \max J \leq \abs{\bb Y} \leq h }. 
\end{equation}
The number of summands in the above estimate is $\binom{h}{x} \leq h^{x}$. It follows from repeated application of \eqref{EXC:eq:901} that each of the probabilities on the right hand side of \eqref{EXC:eq:902} is at most $(2M/\e n)^{x}$. It remains to set $x = \abs{E}$.

We next show that $\bb Y$ is very unlikely to have cardinality much larger than $\log n$. 
Let $k \geq 1$ be an integer. If $\abs{\bb Y} > k\abs{E}$ then there exist $u,v$ with $1 \leq u < v \leq n+1$ and $v-u = k$ such that $\sum_{i=u}^{v-1} \bb a_i < n$. Hence, by Hoeffding's inequality (following the same argument as in Lemma \ref{EXC:lem:kill-bad-U})
\begin{equation}\label{EXC:eq:903}
	\PP\brabig{ \abs{\bb Y} > (2\abs{E}+6)\abs{E} \log n } = O(1/n^{\abs{E}}). 
\end{equation}
Combining \eqref{EXC:eq:901} and \eqref{EXC:eq:903} we conclude that
\begin{equation}\label{EXC:eq:910}
	\PP\brabig{ \cA_G(\bb a, \bb w ) } 
	\leq \PP\brabig{ \abs{\bb X} = \abs{E} }
	= O_{M} \bra{ \bra{ \frac{\log n}{\e n}}^{\abs{E}} }.
\end{equation}

Lastly, we exploit the assumption that $G$ contains a path of length $\geq 2$. It implies that $\abs{E} \geq \abs{F} + 1$, and in general we have $\abs{E}+ \abs{F} = N < 2M$. Hence, as long as $n$ is large enough, we have
\begin{equation}\label{EXC:eq:911}
	\bra{\frac{\log n}{\e n}}^{\abs{E}} \leq \bra{\frac{\log n}{\e n}}^{\abs{F}+1} \leq \bra{\frac{\log n}{\e}}^{M}/n^{\abs{F}+1},
\end{equation}
and \eqref{EXC:eq:75Z} follows by combining \eqref{EXC:eq:910} and \eqref{EXC:eq:911}.
\end{proof} 
\section{Lower bound}\label{section:LOW}

\subsection{} In this section we give a lower bound for the maximal possible number of distinct consecutive sums corresponding to a permutation of a given size. Throughout, as usual, $n$ denotes a positive integer and all instances of the $o(\cdot)$ notation refer to the limit $n \to \infty$.

It follows from Theorem \ref{thm:INT:B} that
\begin{equation}\label{eq:LOW:001}
	\max_{a \in \Sym([n])} \abs{S(a)} \geq \bra{ \frac{1+e^{-2}}{4} +o(1)} n^2,
\end{equation}
which is asymptotically better than the bound from Proposition \ref{I:prop:cexple-basic}. Presently, we prove a slightly stronger estimate
\begin{equation}\label{eq:LOW:002}
\max_{a \in \Sym([n])} \abs{S\bra{a})} \geq \bra{ \frac{3}{2}-\frac{2}{\sqrt{e}} + o(1)} n^2,
\end{equation}
which is the lower bound in Theorem \ref{thm:INT:A}.

We note that the numerical values of the two constants appearing in \eqref{eq:LOW:001} and \eqref{eq:LOW:002} are rather close: 
\[ \frac{3}{2}-\frac{2}{\sqrt{e}} = 0.286\dots, \qquad \frac{1+e^{-2}}{4} = 0.283\dots. \]
Our main goal here is to show that \eqref{eq:LOW:001} is not optimal, and we anticipate that further improvements are possible.

\subsection{}
In order to prove \eqref{eq:LOW:002}, we consider a randomised variant of the construction used in the proof of Proposition \ref{I:prop:cexple-basic}. We keep the constraint that sums of pairs of consecutive entries in the permutation should be constant, but we no longer insist on any particular ordering of these pairs. More precisely, throughout this section we let $\bb a$ denote a permutation of $[n]$ chosen uniformly at random subject to the constraint that 
\begin{equation}\label{eq:LOW:502}
\bb a_{i} + \bb a_{i+1} = n+1
\end{equation} 
for each odd integer $i$ with $1 \leq i < n$. Note that if $n$ is odd then necessarily $\bb a_{n} = (n+1)/2$; in this case, it will be convenient to additionally let $\bb a_{n+1}$ denote $(n+1)/2$, so that \eqref{eq:LOW:502} holds also for $i = n$.

\begin{remark}\label{rmrk:LOW:001}
The alert Reader will have noticed that we did not compute $\abs{S(a)}$ for the permutation $a$ appearing in the proof of Proposition \ref{I:prop:cexple-basic}; we merely obtained a lower bound. Finding precise estimates on $\abs{S(a)}$ is simple but mundane, and it does not lead to better bounds than \eqref{eq:LOW:002}, so we do not pursue this issue further.
\end{remark}

We will obtain an asymptotic formula for the number of distinct consecutive sums corresponding to the  the random permutation $\bb a$ we have just defined. The estimate \eqref{eq:LOW:002} is a direct consequence of the following result.

\begin{proposition}\label{prop:LOW:001}
	Let $n$ and $\bb a$ be as introduced above. Then
	\begin{equation}\label{eq:LOW:003}
		\EE \abs{S\bra{\bb a}} = \bra{ \frac{3}{2}-\frac{2}{\sqrt{e}} + o(1)} n^2.
	\end{equation}
\end{proposition}

We follow a strategy of proof which is roughly analogous to the proof of Theorem \ref{EXC:prop:P-s-in-S}, although the argument is significantly simpler.

We derive Proposition \ref{prop:LOW:001} from a statement concerning the probability that a single putative sum belongs to $S(\bb a)$. This is analogous to how Proposition \ref{EXP:prop:main} is derived from Proposition \ref{EXC:prop:P-s-in-S}. For an integer $s$ and a rational number $x$, we say that $s$ is divisible by $x$ if $x \neq 0$ and $s/x$ is an integer.

\begin{proposition}\label{LOW:prop:P-s-in-S}
Let $n$ and $\bb a$ be as introduced above. Let $s$ be an integer with $1 \leq s \leq \binom{n+1}{2}$ and put $\sigma = s/\binom{n+1}{2}$. If $s$ is divisible by $(n+1)/2$ then $s \in S(\bb a)$. Otherwise, 
	\begin{equation}\label{eq:LOW:004}
	\PP(s \not \in S(\bb a)) = e^{-\ffrac{(1-\sigma)}{2}} \sigma + o(1).
	\end{equation}
\end{proposition}

\begin{proof}[Proof of Proposition \ref{prop:LOW:001} assuming Proposition \ref{LOW:prop:P-s-in-S}]
It follows from Proposition \ref{LOW:prop:P-s-in-S} that
\begin{align*}
	\EE \abs{S(\bb a)} 
	 &= n^2\bra{ \frac{1}{2} + o(1) }\int_0^{1} \bra{1 - e^{-\ffrac{(1-\sigma)}{2}}\sigma} d\sigma 
	\\&= n^2 \bra{\frac{3}{2}-\frac{2}{\sqrt{e}} + o\bra{1} }.\qedhere
\end{align*}
\end{proof}

\subsection{} We devote the remainder of this section to proving Proposition \ref{LOW:prop:P-s-in-S}. From this point, fix the integer $s$ with $1 \leq s \leq \binom{n+1}{2}$. We can write $s$ in the form $s = l(n+1) + k$, where $0 \leq l < n/2$ and $1 \leq k < n+1$. We also put $\lambda = 2 l/n \in [0,1)$ and $\kappa = k/(n+1) \in [0,1)$. Note that, with the notation from Proposition \ref{LOW:prop:P-s-in-S}, $\sigma = s/\binom{n+1}{2} = \lambda + O(1/n)$. We will frequently use the basic observation that for any integers $u,v$ with $1 \leq u < v \leq n+1$ we have
\begin{equation}\label{eq:LOW:005}
\sum_{i=u}^{v-1} \bb a_i = 
\begin{cases}
	(v-u)(n+1)/2 & \text{ if } u \equiv 1 \bmod 2 \text{ and } v \equiv 1 \bmod 2, \\
	(v-u)(n+1)/2 + \bb a_{u} - \bb a_{v} & \text{ if } u \equiv 0 \bmod 2 \text{ and } v \equiv 0 \bmod 2, \\
	(v-u-1)(n+1)/2 + \bb a_u & \text{ if } u \equiv 0 \bmod 2 \text{ and } v \equiv 1 \bmod 2, \\
	(v-u+1)(n+1)/2 - \bb a_{v} & \text{ if } u \equiv 1 \bmod 2 \text{ and } v \equiv 0 \bmod 2.
\end{cases}
\end{equation}

We first address the case where $s$ is divisible by $(n+1)/2$. If $s$ is divisible by $n+1$ then $s = \sum_{i=1}^{2l} \bb a_i$ so $s \in S(\bb a)$. If $s$ is divisible by $(n+1)/2$ but not by $n$ then $n$ is necessarily odd and $s = \sum_{i=n-2l}^n \bb a_i$, so again $s \in S(\bb a)$. Hence, from now on we may assume that $s$ is not divisible by $(n+1)/2$, meaning that $k \neq 0,(n+1)/2$.

It follows from \eqref{eq:LOW:005} that if $s = \sum_{i=u}^{v-1} \bb a_i$ for some integers $u,v$ then $v-u \in \{ 2l, 2l+1, 2l+2\}$. Let $n_1 := \ceil{n/2}-l$, and for each $u$ with $1 \leq u \leq n_1$ let $\cA_u$ denote the event that $s$ is the sum of $2l$ consecutive entries of $\bb a$ starting at $2u$, that is
\begin{equation}\label{eq:LOW:010}
	\cA_u \iff s = \sum_{i=2u}^{2u+2l-1} \bb a_i \iff \bb a_{2u} - \bb a_{2u+2l} = k.
\end{equation}
Likewise, let $n_2 :=  \ceil{n/2}-l-1$ and for $u$ with $1 \leq u \leq n_2$ let $\cB_u$ be the analogously defined event for sums of length $2l+2$, that is
\begin{equation}\label{eq:LOW:011}
	\cB_u \iff s = \sum_{i=2u}^{2u+2l+1} \bb a_i \iff \bb a_{2u} - \bb a_{2u+2l+2} = k-(n+1).
\end{equation}
Finally, note that if $s$ is the sum of $2l+1$ consecutive entries of $\bb a$, that is, $s = \sum_{i=u}^{u+2l} \bb a_i$ for some $u$, then either $u$ is even and $a_u = k$ or $u$ is odd and $a_{u+2l} = k$.  We let
$\cC$ denote the event that $s$ is the sum of $2l+1$ consecutive entries of $\bb a$, that is $\cC$ holds if and only if the (unique) integer $i$ with $\bb a_i = k$ satisfies either $i \equiv 0 \bmod{2}$ and $1 \leq i \leq n-2l$, or $i \equiv 1 \bmod{2}$ and $2l < i \leq n$. Note that the latter condition is always true if $n$ is odd and $s$ is an odd multiple of $(n+1)/2$. 

The events introduced above are defined so that  $s \in S(\bb a)$ if and only if at least one of $\cA_u$ ($1 \leq u \leq n_1$), $\cB_u$ ($1 \leq u \leq n_2$) and $\cC$ is true. An elementary computation shows that
\begin{align*}
	\PP( \cA_u) &= \kappa/n + O(1/n^2)  & (1 \leq u \leq n_1), \\
	\PP( \cB_u) &= (1-\kappa)/n + O(1/n^2) & (1 \leq u \leq n_2),  \\
	\PP(\cC) &= 1-\lambda + O(1/n),
\end{align*}
We will derive Proposition \ref{LOW:prop:P-s-in-S} from the following approximate independence condition, much like we derived Proposition \ref{EXC:prop:P-s-in-S} from Proposition \ref{EXC:prop:P-wedge-s-in-S}.

\begin{proposition}\label{prop:LOW:101} Fix integers $K,L \geq 0$ and put $M = K+L$.
Let $n$, $\bb a$ and $s$ be as introduced above. Then
\begin{equation}\label{eq:LOW:020}
	\sum_{\abs{I} = K} \sum_{\abs{J}=L} \PP\braBig{ \bigwedge_{i \in I} \cA_i \wedge \bigwedge_{j \in J} \cB_j } = \frac{\kappa^K (1-\kappa)^L (1-\lambda)^M}{K! L! 2^M} + O_{M}\bra{\ffrac{1}{n}},
\end{equation}
where the sums are taken over all sets $I \subset [n_1]$, $J \subset[n_2]$ with cardinalities $K$ and $L$ respectively. Likewise,
\begin{equation}\label{eq:LOW:021}
	\sum_{\abs{I} = K} \sum_{\abs{J}=L} \PP\braBig{ \bigwedge_{i \in I} \cA_i \wedge \bigwedge_{j \in J} \cB_j \wedge \cC } 
	= \frac{\kappa^K (1-\kappa)^L (1-\lambda)^{M+1}}{K! L! 2^M} + O_{M}\bra{\ffrac{1}{n}}.
\end{equation}
Above, the constants implicit in the error terms depend only on $M$.
\end{proposition}
\begin{proof}[Proof of Proposition \ref{LOW:prop:P-s-in-S} assuming Proposition \ref{prop:LOW:101}] Let $N \geq 0$ be and odd integer.  Applying the Bonferroni inequality and substituting the bounds from Proposition \ref{prop:LOW:101} we obtain
\begin{align*}
\PP\bra{s \not \in S(\bb a)} &
\leq \sum_{M=0}^{N} \bra{-1}^{M}  \sum_{K+L = M} \sum_{\abs{I} = K}  \sum_{\abs{J} = L}  \PP\bra{\bigwedge_{i \in I} \cA_i \wedge \bigwedge_{j \in J} \cB_j} 
	\\& - 
	\sum_{M=0}^{N-1} \bra{-1}^{M}  \sum_{K+L = M} \sum_{\abs{I} = K}  \sum_{\abs{J} = L}  \PP\bra{\bigwedge_{i \in I} \cA_i \wedge \bigwedge_{j \in J} \cB_j \wedge \cC} 
\\ &= \sum_{M=0}^{N} \bra{-1}^{M} \sum_{K+L = M} \frac{\kappa^K (1-\kappa)^L (1-\lambda)^M}{K! L! 2^M} 
 \\ &- \sum_{M=0}^{N-1} \bra{-1}^{M} \sum_{K+L = M} \frac{\kappa^K (1-\kappa)^L (1-\lambda)^{M+1}}{K! L! 2^M} 
+ O_N\bra{\ffrac{1}{n}} 
 \\& = 
 \sum_{M=0}^{N} \bra{-1}^{M}  \frac{\lambda \bra{1-\lambda}^{M} }{2^M M!}
  + O\bra{\ffrac{1}{N!}} + O_N\bra{\ffrac{1}{n}} 
 \\& = e^{-\ffrac{(1-\lambda)}2} \lambda + O\bra{\ffrac{1}{N!}} + O_N\bra{\ffrac{1}{n}}, 
\end{align*}
where as usual we use the notation $O_N\bra{\cdot}$ to signify that the implicit constant is allowed to depend on $N$. Letting $N \to \infty$ slowly with $n$, we conclude that
\begin{equation}\label{eq:LOW:120}
\PP\bra{s \not \in S(\bb a)} \leq e^{-\ffrac{(1-\lambda)}{2}}\lambda + o(1),
\end{equation}
with the error term uniform with respect to the choice of $s$. A symmetric argument yields the inequality in the reverse direction.
\end{proof}

\subsection{} In order to finish the proof of Proposition \ref{prop:LOW:001} it remains to prove Proposition \ref{prop:LOW:101}. This task is naturally separated into two steps. Firstly, we obtain a uniform bound on the probabilities in \eqref{eq:LOW:020} and \eqref{eq:LOW:021}. Secondly, we obtain a more accurate estimate under suitable genericity conditions. We only prove \eqref{eq:LOW:020}; the proof of \eqref{eq:LOW:021} is analogous and requires no new ideas.

In the remainder of this section, we will use notation from Proposition \ref{prop:LOW:001}. Also, let $I \subset [n_1]$ and $J \subset [n_2]$ be sets with $\abs{I} = K$ and $\abs{J} = L$ and let 
\[ I = \{i_1 < i_2 < \dots < i_K\} \subset [n_1], \qquad J = \{j_1 < j_2 < \dots < j_L\} \subset [n_2] \]
be their increasing enumerations. We also put $i_m' := i_m + l$ ($1 \leq m \leq K$), $j_m' := j_m + l+1$ ($1 \leq m \leq L$) and
\[ I' := \{i'_1 < i_2' < \dots < i_K'\}, \qquad J' := \{j_1' < j_2' < \dots < j_L'\}.\]
Recall that for $i \in [n_1]$ and $j \in [n_2]$, the events $\cA_i$ and $\cB_j$ are determined by $\bb a_{2i}, \bb a_{2i+2l}$ and $\bb a_{2j}, \bb a_{2j+2l+2}$ respectively.
It will be convenient to encode the possible relations between the entries of $\bb a$ relevant for different events among $\cA_i$ ($i \in I$) and $\cB_j$ ($j \in J$) by introducing an (undirected) graph $G = G_{I,J}$ on the vertex set $V = I \cup J \cup I' \cup J'$ with edge set $E$ consisting of all pairs$\{i_m,i_m'\}$ ($1 \leq m \leq K$) and $\{j_m,j_m'\}$ ($1 \leq m \leq L$). 

\begin{lemma}\label{lem:LOW:002}
	With notation introduced above, \begin{equation}\label{eq:LOW:031}
	\PP\braBig{ \bigwedge_{i \in I} \cA_i \wedge \bigwedge_{j \in J} \cB_j } \leq 1/n^M + O_M(1/n^{M+1}).
\end{equation}
\end{lemma}
\begin{proof}
If $\PP\brabig{ \bigwedge_{i \in I} \cA_i \wedge \bigwedge_{j \in J} \cB_j } = 0$ then we are done, so suppose this is not the case. Then $I \cap J = \emptyset$, since if $\cA_u$ and $\cB_u$ both hold for some $u$ then 
\[
	s = \sum_{i=2u}^{2u+2l} \bb a_i =  \sum_{j=2u}^{2u+2l+2} \bb a_j, 
\]
which is absurd. By the same token, $I' \cap J' = \emptyset$. It follows that the corresponding graph $G = G_{I,J} = (V,E)$ is the union of a number of paths with $M$ edges in total. 

Let $F \subset V$ denote the set of endpoints (i.e.\ largest elements) of the paths comprising $G$. Given any sequence $(c_i)_{i \in F}$ taking values in $[n]$, one can construct a sequence $(c_i')_{i \in V \setminus F}$ such that if $\bb a_{2i} = c_i$ for all $i \in F$ and $\bigwedge_{i \in I} \cA_i \wedge \bigwedge_{j \in J} \cB_j$ holds then $\bb a_{2i} = c_i'$ for all $i \in V \setminus F$. Note also that $\abs{ V \setminus F} = M$. It follows that for each $c \in [n]^F$ with no repeating entries,
\[
\EE \bra{\PP\braBig{ \bigwedge_{i \in I} \cA_i \wedge \bigwedge_{j \in J} \cB_j } \ \middle| \ (\bb a_{2i})_{i \in F} = c} \leq 1/(n-\abs{F})^M.
\]
Since $c$ was arbitrary, we conclude that
\begin{align*}
	\PP\braBig{ \bigwedge_{i \in I} \cA_i \wedge \bigwedge_{j \in J} \cB_j } \leq 1/(n-\abs{F})^M,	
\end{align*}
from which \eqref{eq:LOW:031} readily follows.
\end{proof}

\begin{lemma}\label{lem:LOW:003}
	With notation introduced above, if the sets $I,I',J,J'$ are pairwise disjoint and $n,n+1 \not \in I' \cup J'$ then
\begin{equation}\label{eq:LOW:131}
	\PP\braBig{ \bigwedge_{i \in I} \cA_i \wedge \bigwedge_{j \in J} \cB_j } = \bra{\frac{\kappa}{n}}^K \bra{\frac{1-\kappa}{n}}^L + O_M(1/n^{M+1}).
\end{equation}
\end{lemma}
\begin{proof}
	Note that the disjointness of $I,I',J,J'$ is equivalent to the statement that $G_{I,J}$ is a union of disjoint paths of length $1$. For each integer $m$ with $1 \leq m \leq K$ and for each set $F \subset V \setminus \{i_m,i_m'\}$ and each sequence $c \in [n]^F$ with no repeating entries it holds that
\begin{equation}\label{eq:LOW:032}
	\PP\left( \cA_{i_m} \ \middle| \ (\bb a_{2i})_{i \in F} = c \right) = \frac{\kappa}{n} + O\bra{\abs{F}/n^2}.  
\end{equation}
By the same token, for $m$ and $F$ with $1 \leq m \leq L$ and $F \subset V \setminus \{j_m,j_m'\}$ $c \in [n]^F$ with no repeating entries we have
\begin{equation}\label{eq:LOW:033}
	\PP\left( \cB_{j_m} \ \middle| \ (\bb a_{2i})_{i \in F} = c  \right) = \frac{1-\kappa}{n} + O\bra{\abs{F}/n^2}.  
\end{equation}
 Let $I_0,J_0$ be any sets with $I_0 \subset I$ and $J_0 \subset J$, and let 
$\cD$ be an event that is equal to either $\cA_i$ for some $i \in I \setminus I_0$  or to $\cB_j$ for some $j \in J \setminus J_0$.
Combining the two bounds \eqref{eq:LOW:032} and \eqref{eq:LOW:033}, and 
recalling that each of the events $\cA_i,\cB_j$ is determined by two corresponding entries in $\bb a$, we conclude that
\begin{equation}\label{eq:LOW:034}
	\PP\Big(\cD \ \Big| \ \bigwedge_{i \in I_0} \cA_i \wedge \bigwedge_{j \in J_0} \cB_j \Big) = \PP(\cD) + O_M(1/n^2).
\end{equation}
Equation \eqref{eq:LOW:031} now follows by inductive application of \eqref{eq:LOW:034}.
\end{proof}

We now have all the ingredients necessary to finish the argument.

\begin{proof}[Proof of Proposition \ref{prop:LOW:001}]
	Among all possible choices of the pair of sets $I,J$ with $I \subset [n_1],J\subset [n_2]$ and $\abs{I}=K,\abs{J} = L$, pairs such that $n$ or $n+1$ belongs to $I' \cup J'$ or $I,I',J,J'$ are not disjoint constitute proportion $O_M(1/n)$. The number of all pairs $I,J$ as above is $\bra{ (1-\lambda)n/2 }^{M}/K!L! + O_M(1/n^{M+1})$. Proposition \ref{prop:LOW:001} now follows by estimating all sums in \eqref{eq:LOW:021} by Lemma \ref{lem:LOW:003} if applicable and Lemma \ref{lem:LOW:002} otherwise.
\end{proof} 
\section{Upper bound}\label{section:EXT-up}

\subsection{}
Results from previous sections suggest that distinct consecutive sums $\sum_{i=u}^{v-1} a_i$ tend to be rather numerous, where as usual $n$ denotes a large integer and $a \in \Sym([n])$ denotes a permutation. The trivial upper bound on the number of distinct consecutive sums is $\binom{n+1}{2}$, which happens to be both the upper bound for any single sum and the number of distinct intervals. It is natural to ask if this bound is asymptotically sharp, and it turns out that it is not. In this section we obtain a slight improvement, namely
\begin{equation}
\max_{a \in \Sym([n])} {S(a)} \leq \bra{\frac 14 + \frac{\pi}{16} + o(1)}{n^2},
\label{EXTup::eq:L-def}
\end{equation}
thus proving the upper bound in Theorem \ref{thm:INT:A}.

It will be helpful to consider the set of sums which are above average value, defined for $a \in \Sym([n])$ by
\begin{equation}
	L(a) := \bbra{ (u,v)  \ : \ 1 \leq u < v \leq n+1, \ \sum_{i=u}^{v-1} a_i > \frac{1}{2} \binom{n+1}{2} }.
	\label{EXTup:eq:La-def}
\end{equation}
	Our main idea is to show that $L(a)$ can never have size close to the trivial upper bound of $ \frac{1}{2} \binom{n+1}{2}$. The following proposition easily implies \eqref{EXTup::eq:L-def}.
	
\begin{proposition}\label{EXTup:prop:L(a)-bound}
	Let $n \geq 1$ be an integer and let $a \in \Sym([n])$. Then
	\[
		\abs{L(a)} \leq \bra{\frac{\pi}{16} + o(1)} n^2.
	\]
\end{proposition}
\begin{proof}[Proof of \eqref{EXTup::eq:L-def} assuming Proposition \ref{EXTup:prop:L(a)-bound}]
	Dividing $S(a)$ into the elements that are larger than $\frac{1}{2} \binom{n+1}{2}$ and those that are not, we easily find that
\begin{equation*}
\abs{S(a)} \leq \frac{1}{2} \binom{n+1}{2} + \abs{L(a)} \leq\bra{ \frac{1}{4} + \frac{\pi}{16} + o(1)}n^2.
\end{equation*} 
This is precisely the required bound.
\end{proof}

\begin{remark}
	The constant $\ffrac{\pi}{16}$ in Proposition \ref{EXTup:prop:L(a)-bound} cannot be improved, as shown by the ``tent map'' permutation:
	\[
		a_i = 
		\begin{cases}
		2 i  & \text{ if } i \leq \frac{n}{2}, \\
		2(n-i) + 1 & \text{ if }  i > \frac{n}{2}.
		\end{cases}
	\]
	This is essentially the only possible example, as will become clear in the course of the proof. However, essentially the same arguments as for the identity permutation, one can check that for the permutation $a$ defined above we have $S(a) = o(n^2)$; see also Example \ref{exple:END:bounded-complexity}. This leads us to believe that the upper bound in Theorem \ref{I:prop:upper-bound} is not sharp.
\end{remark}

\subsection{}
We will devote the remainder of this section to proving Proposition \ref{EXTup:prop:L(a)-bound}. To begin with, we reduce the problem to the case when the permutation can be partitioned into two monotonous parts. This part of the argument restricts the domain significantly, and will play an important role in enabling us to pass to a continuous version of the problem.

We will say that a sequence $(b_j)_{j=1}^m$ is \emph{bitonic} if there exists an integer $j$ with $1 \leq j_0 \leq m$ such that $b_j$ is increasing for $1 \leq j \leq j_0$ and decreasing for $j_0 \leq j \leq m$. Likewise, a function $f \colon [0,1] \to \RR$ is bitonic if there exists $x_0 \in [0,1]$ such that $f$ is increasing on $[0,x_0]$ and decreasing on $[x_0,1]$.  We call $j_0$ and $x_0$ \emph{bitonic points} for $b$ and $f$ respectively. Throughout, we take ``increasing''  to mean ``non-strictly increasing'' and likewise for ``decreasing''. The notions of a strictly bitonic sequence and a strictly bitonic sequence are defined analogously.

\begin{lemma}\label{EXTup:obs:perm-mono}
	Let $n \geq 1$ and $a \in \Sym([n])$. Then there exists $a' \in \Sym([n])$ such that  $\abs{L(a)} \leq \abs{L(a')}$ and $a'$ has a bitonic point $k$ with $\sum_{i=1}^k a_i', \sum_{i=k}^n a_i' \geq \frac{1}{2} \binom{n+1}{2}$.
\end{lemma}
\begin{proof}
	There exists unique $k$ with $1 \leq k \leq n$ such that $\sum_{i=1}^{k} a_i > \frac{1}{2} \binom{n+1}{2}$ and $\sum_{i=k}^{n} a_i \geq \frac{1}{2} \binom{n+1}{2}$. Note that $\sum_{i=1}^{k-1} a_i \leq \frac{1}{2} \binom{n+1}{2}$ and $\sum_{i=k+1}^{n} a_i < \frac{1}{2} \binom{n+1}{2}$.

	Consider the permutation $a'$ obtained from $a$ by sorting $a_1,\dots,a_k$ in increasing order and $a_{k+1},\dots,a_n$ in the decreasing order. More precisely, let $a'$ be such that 
\begin{align*}
\{ a_i' \ : \ 1 \leq i \leq k\} &= \{ a_i \ : \ 1 \leq i \leq k\} \text{ and } a_1' < \dots < a_k',\\
\{ a_i' \ : \ k < i \leq n\} &= \{ a_i \ : \ k < i \leq n\} \text{ and } a_{k+1}' > \dots > a_n'.
\end{align*}
Clearly, $k$ is a bitonic point of $a'$.
	
Take any $(u,v) \in L(a)$. By the choice of $k$, we have $u \leq k < v$. Hence,
	\[ 
		\frac{1}{2} \binom{n+1}{2} \leq \sum_{i=u}^{v-1} a_i 
		= \sum_{i=u}^{k} a_i + \sum_{i=k+1}^{v-1} a_i
		\leq \sum_{i=u}^{k} a_i' + \sum_{i=k+1}^{v-1} a_i' = \sum_{i=u}^{v-1} a_i'. 
	\]	
	Thus, $(u,v) \in L(a')$, and since $(u,v) \in L(a)$ was arbitrary, $L(a) \subseteq L(a')$. In particular, $\abs{L(a)} \leq \abs{L(a')}$.
\end{proof}

\subsection{}
We are now ready to introduce the continuous variant of the problem. The analogue of the space of all permutations of $[n]$ obeying the monotonicity condition in Lemma \ref{EXTup:obs:perm-mono} is the family $\cFm$ of measurable functions $f\colon [0,1] \to [0,1]$ obeying the following conditions:
\begin{enumerate}[label={$(\mathrm{F}_{\arabic*})$},ref={$\mathrm{F}_{\arabic*}$}]
	\item\label{EXTup:cond:01} for any measurable set $E \subset [0,1]$, $\int_{E} f(x)dx \geq \frac{1}{2} {\abs{E}^2}$ and $\int_0^1 f(x) dx = \frac{1}{2}$,
	\item\label{EXTup:cond:02} the function $f$ has a bitonic point $\kappa=\kappa_f$ with $\int_{0}^\kappa f(x) dx = \int_{\kappa}^1 f(x) dx = \frac{1}{4}$.
\end{enumerate}	
(Here and elsewhere, if $E \subset \RR$ is measurable then $\abs{E}$ denotes the Lebesgue measure of $E$.) Note that for a permutation $a$ and any index set $I$ we have $\sum_{i \in I} a_i \geq \binom{\abs{I}+1}{2}$, in analogy to condition \eqref{EXTup:cond:01}.

We will also occasionally need to use the larger family $\cF$ of functions $f\colon [0,1] \to [0,1]$ which only satisfy the condition \eqref{EXTup:cond:01} but not necessarily \eqref{EXTup:cond:02}. We note in passing that $\cF$ is convex, and both $\cFm$ and $\cF$ are closed in the $L^1$ topology; from this point, we endow $\cFm$ and $\cF$ with the topology induced from $L^1$. 

Another component needed for the continuous variant of the problem is an analogue of $L(a)$ from \eqref{EXTup:eq:La-def}. For any $f \in \cF$, let
\begin{equation}
	L(f) := \bbra{(x,y) \in [0,1]^2 \ : \ x \leq y,\ \int_x^y f(t) dt \geq \frac{1}{4} }.
\end{equation}
The continuous analogue of Proposition \ref{EXTup:prop:L(a)-bound} is the following statement.

\begin{proposition}\label{EXTup:prop:L(f)-bound}
	Suppose that $f \in \cFm$. Then $\abs{L(f)} \leq \frac{\pi}{16}$.
\end{proposition}

This bound is sharp. The (essentially unique) function $f \in \cFm$ with $\Lambda(f) = \frac{\pi}{16}$ will turn out to be the ``tent map'':
\[
		f(x) = 
		\begin{cases}
			2x &  \text{ if } x \leq \frac{1}{2},\\
			2(1-x) &  \text{ if } x \geq \frac{1}{2}.
		\end{cases}
\]

\subsection{}
We defer the proof of  Proposition \ref{EXTup:prop:L(f)-bound}; our immediate goal is to show that it implies Proposition \ref{EXTup:prop:L(a)-bound}. Before we do that, we make some preliminary observations which will be useful in the course of this deduction, as well as in the main body of the argument proving Proposition \ref{EXTup:prop:L(f)-bound}. For $f \in \cF$, define 

\begin{align}
		v_f(u) & := \sup\left\{ v \in [0,1] \ : \ \int_{u}^v f(x) dx \leq \frac{1}{4} \right\}, \label{eq:UP:def-of-v} \\
		u_f(v) & := \inf\left\{ u \in [0,1] \ : \ \int_{u}^v f(x) dx \leq \frac{1}{4} \right\}, \label{eq:UP:def-of-u} \\
		\Lambda(f) &:= \abs{L(f)}. \label{eq:UP:def-of-L}
\end{align}
Hence, if $f \in \cFm$ then $\int_{u}^{v_f(u)} f(x) dx = \frac{1}{4}$ for $u \in [0,\kappa_f]$ and $v_f(u) = 1$ for $u \in [\kappa_f,1]$, and similar relations hold for $u_f$.

\begin{lemma}\label{EXTup:obs:prelim}
	With definitions as above, the following are true.
	\begin{enumerate}
	\item\label{item:53:A} For any $f \in \cF$ we have $v_f, u_f \in L^1([0,1])$ and the maps $ f \mapsto v_f$ and $f \mapsto u_f$ from $\cF$ to $L^1([0,1])$ are continuous.
	\item\label{item:53:B} For any $f \in \cF$ and any $u_0,v_0 \in [0,1]$ with $\int_{u_0}^{v_0} f(x) dx = \frac{1}{4}$ we have
	\[ \Lambda(f) = \int_{0}^{u_0} (1-v_f(x)) dx + \int_{v_0}^1 u_f(x)dx- u_0 (1-v_0).\]
	\item\label{item:53:C} We have the formulas \[\Lambda(f) = \int_{0}^1 (1-v_f(u))du = \int_{0}^1 u_f(v) dv.\]
	\item\label{item:53:D} The map $ f \mapsto \Lambda(f)$ from $\cF$ to $\RR$ is continuous.
	\item\label{item:53:E} For any $f \in \cFm$ the set $L(f)$ is convex.
	\end{enumerate}
\end{lemma}	

\begin{proof}
	We begin with \eqref{item:53:A}. We only prove continuity of $f \mapsto v_f$, the argument for $f \mapsto u_f$ is analogous. Take any $f,f_n \in \cF$ with $f_n \to f$ in $L^1$. Fix $u$ and let $v = v_f(u)$. Let us suppose that $v < 1$, since the case $v = 1$ is easier. For any $\delta > 0$ we have
	\[
		\int_u^{v-\delta} f(x) dx + \frac{1}{4} \delta^2 < 
		\int_u^{v} f(x) dx = \frac{1}{4} < \int_u^{v+\delta} f(x) dx -  \frac{1}{4} \delta^2
	\]
	Thus, there exists $n_0 = n_0(\delta)$ such that for $n > n_0(\delta)$ we have	
	\[
		\int_u^{v-\delta} f_n(x) dx + \frac{1}{8} \delta^2 <   \frac{1}{4} < \int_u^{v+\delta} f_n(x) dx -  \frac{1}{8} \delta^2,
	\]
	and consequently $v-\delta < v_{f_n}(u) < v + \delta$. Taking $\delta \to 0$ we conclude that $v_{f_n}(u) \to v$ as $n \to \infty$. Hence, $v_{f_n} \to v_f$ pointwise, and since all relevant functions are bounded, also $v_{f_n} \to v_f$ in $L^1$ as $n \to \infty$

	 The integral formula \eqref{item:53:B} for $\Lambda(f)$ follows from partitioning $L(f)$ into three parts: $ L_- = L(f) \cap \{ u < u_0,\ v< v_0 \}$, ${L}_+ = {L}(f) \cap \{ u > u_0,\ v > v_0 \}$ and ${L}_* = {L}(f) \cap \{ u \leq u_0,\ v \geq v_0 \} = [0,u_0] \times [v_0,1]$.{ The formulas \eqref{item:53:C} are special cases of \eqref{item:53:B} with where $(u_0,v_0)$ is $(0,\kappa_f)$ or $(\kappa_f,1)$.} Continuity of $f \mapsto \Lambda(f)$ in \eqref{item:53:D} is a direct consequence of the previous points \eqref{item:53:A} and \eqref{item:53:C}. 
	
	Finally, we prove convexity of $L(f)$ in \eqref{item:53:E}. Suppose that $(u_1,v_1), (u_2,v_2) \in L(f)$ and let $u = \ffrac{\bra{u_1+u_2}}{2}, v= \ffrac{\bra{v_1+v_2}}{2}$. We may assume without loss of generality that $\int_{u_1}^{v_1}f(x)dx = \int_{u_2}^{v_2}f(x)dx = \frac{1}{4}$ (which implies that $u_1, u_2 \leq \kappa_f \leq v_1,v_2$) and that $u_1 \leq u_2$ (which implies that $v_1 < v_2$). Then 
	\[
		\int_{u_1}^{u_2}f(x)dx = \int_{v_1}^{v_2}f(x)dx =: I.
	\]
	and because of monotonicity of $f$ on the relevant intervals, we have 
	\[\int_{u_1}^{u}f(x)dx \leq \frac{1}{2} I \leq \int_{u}^{u_2}f(x)dx \quad \text{ and } \quad \int_{v_1}^{v}f(x)dx \geq \frac{1}{2}I \geq \int_{v}^{v_2}f(x)dx.\]
	It follows that
	\[  \int_{u}^{v}f(x)dx =  \int_{u_1}^{v_1}f(x)dx -  \int_{u_1}^{u}f(x)dx +  \int_{v_1}^{v}f(x)dx \geq \frac{1}{4},\]
	and consequently $(u,v) \in L(f)$. Since $L(f)$ is closed and $(u_1,v_1), (u_2,v_2) \in L(f)$ were arbitrary, this proves convexity.
\end{proof}

\begin{proof}[Proof of Proposition \ref{EXTup:prop:L(a)-bound} assuming Proposition \ref{EXTup:prop:L(f)-bound}]
	 For each $n \geq 1$, let $a^{(n)}$ denote the permutation of $[n]$ which maximizes $\abs{L(a^{(n)})}$.  By Lemma \ref{EXTup:obs:perm-mono} we may assume without loss of generality that $a^{(n)}_i$ are bitonic and have bitonic points $k^{(n)}$ with 
	 \begin{equation}\label{eq:EXC:953-a}	 
	 \sum_{i=1}^{k^{(n)}} a_i^{(n)} \geq \frac{1}{2} \binom{n+1}{2} 
	 \text{ and } \sum_{i=k^{(n)}}^n a_i^{(n)} \geq \frac{1}{2} \binom{n+1}{2}.
	 \end{equation} 	
	 	
	For any $n \geq 1$, we associate to $a^{(n)}$ the step function $f_n\colon[0,1] \to \RR$ defined by
\begin{equation}\label{eq:EXC:def-of-f_n}
 	f_n\bra{\frac{i + t}{n}} = \frac{a_i^{(n)}}{n+1} \quad \text{ for } i \in [n] \text{ and } t \in [-1,0), 
\end{equation} 	
and (for completeness) $f(1) = \frac{a_n^{(n)}}{n+1}$. 
	 It follows directly from the definition \eqref{eq:EXC:def-of-f_n} that we have the bounds $0 \leq f_n(x) \leq 1$ and the formula 
\begin{equation}\label{eq:EXC:953}	 
	 \sum_{i=u}^{v-1} a_i^{(n)} = n(n+1) { \int_{(u-1)/n}^{(v-1)/n} f_n(x) dx}
\end{equation} 	
	 for any $u,v$ with $1 \leq u \leq v \leq n+1$. In particular, $\int_0^1 f_n(x) dx = \frac{1}{2}$. 
	 
	 It is not difficult to see that for any measurable set $E \subset [0,1]$ with $\abs{E} = \ffrac{\bra{m + \mu}}{n}$ with $m \in \NN_0,$ and $\mu \in [0,1)$ we have 
	 \[\int_E f_n(x) dx \geq \frac{1 + 2 + \dots + m + (m+1)\mu }{n(n+1)} = \frac{(m+2\mu)(m+1)}{2n(n+1)} \geq \frac{\abs{E}^2}{2},\]
where the last inequality can be checked with elementary methods. It is also clear from the construction that $f_n$ is increasing on $\left[0,\ffrac{k^{(n)}}{n}\right)$ and decreasing on $\left[\ffrac{\bra{k^{(n)}-1}}{n},1\right]$, and in particular any $\kappa$ with $\ffrac{\bra{k^{(n)}-1}}{n} \leq \kappa < \ffrac{k^{(n)}}{n}$ is a bitonic point for $f_n$. It now follows from \eqref{eq:EXC:953-a} and \eqref{eq:EXC:953} that  $f_n \in \cFm$.

	 By Proposition \ref{EXTup:prop:L(f)-bound}, $\Lambda(f_n) \leq \frac{\pi}{16}$. It remains to relate $\Lambda(f_n)$ to $\abs{L(a^{(n)})}$. For any $u,v$ with $1 \leq u < v \leq n+1$ it follows from \eqref{eq:EXC:953} that
	 \[(u,v) \in L\left(a^{(n)}\right) \iff \left(\frac{u-1}{n}, \frac{v-1}{n} \right) \in L(f_n).\]
	 Hence, the number of points in $L(f_n)$ on the lattice $\frac{1}{n}\ZZ \times \frac{1}{n} \ZZ$ is precisely $\abs{L\left(a^{(n)}\right)}$. Since $L(f_n)$ is convex by Lemma \ref{EXTup:obs:prelim}.\eqref{item:53:E}, $ \Lambda(f_n) = \ffrac{\abs{L(a^{(n)})}}{n^2} + O(1/n)$, whence 
	\[
		\max_{a \in \Sym([n])} \abs{L\left(a\right)} = \abs{L\left(a^{(n)}\right)} \leq \Lambda(f_n)n^2 + O(n) \leq \bra{\frac{\pi}{16} + O(1/n)}n^2,
	\]
	which finishes the proof.
\end{proof}
	
\subsection{}
	The rest of this section will be devoted to proving Proposition \ref{EXTup:prop:L(f)-bound}. Our first step in that direction is to show that the supremum of $\Lambda(f)$ for ${f \in \cF}$ is realised by a function $f_{*}$ in $\cFm$. It is convenient to allow $f$ to range over the larger family $\cF$ to simplify perturbation arguments later on. 

\begin{lemma}
	There exists a function $f_{*} \in \cFm$ such that $\displaystyle \Lambda(f_{*}) = \sup_{f \in \cF} \Lambda(f)$.
\end{lemma}
\begin{proof}
	For any $f \in \cF$, there exists $g \in \cFm$ such that $\Lambda(g) \geq \Lambda(f)$. This follows from an argument essentially equivalent to the one in Lemma \ref{EXTup:obs:perm-mono}.
	Hence, it will suffice to show that the supremum $\sup_{f \in \cFm} \Lambda(f)$ is realised by some $f_{*}$, which (since $\Lambda$ is continuous by Lemma \ref{EXTup:obs:prelim}) in turn will follow once we show that  $\cFm$ is compact.

	Compactness of $\cFm$ is a direct consequence of the classical Helly's selection theorem, see e.g.\ \cite{Brunk-1956} for details. For a direct proof, consider any sequence $f_n \in \cFm$. Passing to a subsequence, we may assume that $f_n$ converges pointwise on $\mathbb{Q} \cap [0,1]$. By motonicity, $f_n$ converges pointwise a.e.\ to some function $f$. Thus, by the dominated convergence, $f_n$ converges in $L^1$. It is clear that $f \in \cFm$.
\end{proof}

Now that we know that there exists some $f_{*} \in \cFm$ which maximises $\Lambda$, we may study such $f_{*}$ more closely. Note that $f_{*}$ is only defined up to equality almost everywhere. Since $f_{*}$ is bitonic, we may without loss of generality assume that it is strongly upper semicontinuous, meaning that $f(x) = \limsup_{y \to x} f(y)$ for each $x \in [0,1]$. At this point, there is no guarantee that $f_{*}$ is unique (even up to equality almost everywhere); we fix the choice of $f_{*}$ until the end of this section. It comes as no surprise that the behaviour of $\Lambda(f_{*})$ under small distortions is relevant.

\begin{proposition}\label{EXTup:lem:delta-Lambda}
	Let $f \in \cFm$, and suppose that $h \in L^\infty([0,1])$ is such that $f+\tau h \in \cF$ for sufficiently small $\tau > 0$. For sufficiently small $\tau > 0$, let

\begin{align}
	\Delta^{\tau}_h \Lambda(f) &:= \Lambda\bra{f+\tau h} - \Lambda\bra{f},\\
	\delta_h \Lambda(f) &:= \lim_{\tau \to 0+} \frac{1}{\tau} \Delta^{\tau}_h \Lambda(f).
\end{align}
	Then the limit defining $\delta_h \Lambda(f)$ exists and 
	\begin{equation}\label{eq:UP:def-of-w-1}
		\delta_h \Lambda(f) = \int_{0}^1 h(x) w_f(x) dx, 
	\end{equation}
	where $w_f \colon [0,1] \to [0,\infty)$ is given by:
	\begin{equation}\label{eq:UP:def-of-w}		
	w_f(x) := 
	\begin{cases} 
		\int_{0}^x \ffrac{du}{f(v(u))} & \text{ if } x \leq \kappa_f, \\
		\int_{x}^1 \ffrac{dv}{f(u(v))} & \text{ if } x \geq \kappa_f,
	\end{cases}
	\end{equation}
	where $\kappa_f$ is the bitonic point of $f$ appearing in \eqref{EXTup:cond:02}. In particular, $w \colon [0,1] \to \RR_{\geq 0}$ is continuous and strictly bitonic with bitonic point $\kappa_f$.
\end{proposition}

\begin{proof}
We may assume without loss of generality that $\norm{h}_\infty \leq 1$. Following the convention suggested above, for small $\tau > 0$ we define

\begin{align*}
	(\Delta^{\tau}_h u_f)(x) &:= u_{f+\tau h}(x) - u_{f}(x),\
	&(\Delta^{\tau}_h v_f)(x) &:= v_{f+\tau h}(x) - v_{f}(x),
	\\
	(\delta_h u_f)(x) &:= \lim_{\tau \to 0} \frac{1}{\tau} (\Delta^{\tau}_h u_f)(x), 	
	&(\delta_h v_f)(x) &:= \lim_{\tau \to 0} \frac{1}{\tau} (\Delta^{\tau}_h v_f)(x).
\end{align*}
Since $f$ is fixed, we will suppress dependence on $f$, writing $\Delta^{\tau}_h u$, $\delta_h u$, $w$, etc.\ in place of $\Delta^{\tau}_h u_f$, $\delta_h u_f$, $w_f$ whenever ambiguity does not arise.

We have a trivial estimate $\abs{\Delta^{\tau}_h v (u)} \leq 2 \sqrt{ \tau}$ for each $u \in [0,1]$, which follows directly from the chain of inequalities
\begin{equation}\label{eq:UP:482}
\tau  \norm{h}_{1} \geq \abs{
\int_{v(u)}^{v(u)+\Delta^{\tau}_h v(u)} (f-\tau h)(x) dx }
\geq \frac{(\Delta^{\tau}_h v(u))^2}{2} - \tau \norm{h}_{1}.
\end{equation}
(Here and elsewhere, we use the convention that $\int_a^b \equiv -\int_b^a$ if $a > b$). In the same way, we have $\abs{\Delta^{\tau}_h u (v)} \leq 2 \sqrt{ \tau}$ for each $v \in [0,1]$. 

Let us now fix some $\mu < \kappa$. If $\tau$ is small enough, then for $0 \leq u \leq \mu$ we have that $v(u) + \Delta^{\tau}_h v(u) < 1$. For $u < \mu$ we have a refinement of \eqref{eq:UP:482}:

\begin{align}
	\label{EXTup:eq:360}
	\frac{1}{2} &= \int_{u}^{v(u) + \Delta^{\tau}_h v(u)} (f + \tau h)(x) dx 
	\\ &=
	\label{EXTup:eq:361} 
	\frac{1}{2} + \tau \int_u^{v(u) + \Delta^{\tau}_h v(u)} h(x) dx + 
	 \int_{v(u)}^{v(u) + \Delta^{\tau}_h v(u)} f(x) dx
\end{align}
Estimating the integral of $h$ from $v(u)$ to ${v(u) + \Delta^{\tau}_h v(u)}$ trivially, we conclude that
\begin{equation}\label{EXTup:eq:362}
 \frac{1}{\Delta^{\tau}_h v(u)} \int_{v(u)}^{v(u) + \Delta^{\tau}_h v(u)} f(x) dx
 = - \frac{\tau}{\Delta^{\tau}_h v(u)} \bra{
 \int_u^{v(u)} h(x) dx + O\bra{\sqrt{\tau}}
 }.
\end{equation}
	For a.e.\ $u \in [0,\mu]$, the expression on the left hand side tends to $f(v(u))$ as $\tau \to 0$ by the Lebesgue density theorem. Letting $\tau \to 0$ and $\mu \to \kappa$ we conclude that $\delta_h v(u)$ is well defined for a.e.\ $u \in [0,\kappa]$ and
	\begin{equation}
		\delta_h v(u) = \lim_{\tau \to 0}  \frac{\Delta^{\tau}_h v(u)}{\tau} = -\frac{\int_u^{v(u)} h(x) dx }{f(v(u))}. \label{EXTup:eq:54}
	\end{equation}
By a symmetric argument, for a.e.\ $v \in [\kappa, 1]$ we have
	\begin{equation}
		\delta_h v(u) = \frac{\int_{u(v)}^{v} h(x) dx }{f(u(v))}. \label{EXTup:eq:55}
	\end{equation}
Fix once again $0 < \mu < \kappa $ and put $\nu = v(\mu)$. It follows from Lemma \ref{EXTup:obs:prelim} that
\[
	\Lambda(f) = \int_0^{\mu} (1-v(u)) du + \int_{\nu}^1 u(v) dv - \mu (1-\nu).
\]
 Fix also $\e > 0$, sufficiently small that $\nu + \e < 1$. It follows from \eqref{EXTup:eq:360}-\eqref{EXTup:eq:361} (along similar lines as \eqref{EXTup:eq:362}) that for sufficiently small $\tau > 0$ we have the bound
  \[ \abs{ \Delta^{\tau}_h u(v)} \leq \frac{2\tau \norm{h}_1}{ f( \nu + \e)} \] 
 for $u \leq \mu$, and hence $\ffrac{ \Delta^{\tau}_h v(u)}{\tau} $ is uniformly bounded (for a given choice of of $\mu$; the bound does not depend on $\e$). Likewise,  $\ffrac{ \Delta^{\tau}_h u(v)}{\tau} $ if uniformly bounded for $v \geq \nu$.
We may now compute that

\begin{align*}
	\delta_h \Lambda(f) 
	&= \lim_{\tau \to 0} \frac 1\tau \Delta^{\tau}_h \Lambda(f) 
	\\ &= \lim_{\tau \to 0} 
	 \int_0^{\mu} - \frac{ \Delta^{\tau}_h v(u)}{\tau} du + \int_{\nu}^1 \frac{ \Delta^{\tau}_h u(v)}{\tau} dv + \mu \frac{ \Delta^{\tau}_h v(\mu) }{\tau} + O(\tau)
	\\&= -\int_0^{\mu} \delta_h v(u) du + \int_{\nu}^1 \delta_h u(v) dv + \mu \delta_h  v(\mu),
\end{align*}
where the last equality uses the dominated convergence theorem. Passing to the limit $\mu \to 0$ or $\mu \to \kappa$ we find simpler expressions:
\begin{equation}\label{EXTup:eq:56}
	\delta_h \Lambda(f) 
	= -\int_0^{\kappa} \delta_h v(u) du = \int_{\kappa}^1 \delta_h u(v) dv.
\end{equation}

Inserting \eqref{EXTup:eq:54} into the first equation of \eqref{EXTup:eq:56} and exchanging the order of integration, we conclude that
\begin{equation}\label{EXTup:eq:56-a}
	\delta_h \Lambda(f) = \int_0^\kappa \frac{\int_u^{v(u)} h(x) dx }{f(v(u))} du
	= \int_{0}^{1} h(x) \int_{u(x)}^{\min(x,\kappa)} \frac{du}{f(v(u))}dx.
\end{equation}
	Let $w(x)$ denote the value of the inner integral $\int_{u(x)}^x \ffrac{du}{f(v(u))}$. If $x \leq \kappa$ then $u(x) = 0$ so we obtain the sought formula $w(x) = \int_{0}^x \ffrac{du}{f(v(u))}$. The formula $w(x) = \int_{x}^1 \ffrac{dv}{f(u(v))}$ for $x \geq \kappa$ follows by a symmetric argument. Alternatively, the formula also follows from a change of variables $v = v( u)$ in \eqref{EXTup:eq:56-a} together with the observation that $\frac{dv}{du} = \frac{f(u)}{f(v)}$. 
\end{proof}

\subsection{}
	Using standard techniques, we can extract from Proposition \ref{EXTup:lem:delta-Lambda} above some strong structural information about the function $f_{*}$ minimising $\Lambda$. The argument is complicated by the fact that we need to account for a variety of pathological behaviour that $f_{*}$ may a priori exhibit. However, the key idea is simply to relate to each undesirable behaviour of $f_{*}$ a perturbation of $f_{*}$ which increases $\Lambda$.
	
	To avoid cluttering the notation, whenever $f_{*}$ appears in subscript we replace it with $*$; in particular, $\kappa_{*}$ denotes the bitonic point appearing in \eqref{EXTup:cond:02} and $w_{*}$ us the function given by \eqref{eq:UP:def-of-w} for $f=f_{*}$. When this does not lead to ambiguity, we dispense with the subscripts altogether.

\begin{lemma}\label{EXTup:lem:constraints}
	Let $f_{*}$ and $w_{*}$ be as introduced above. Then $f_{*}$ is continuous, strictly bitonic, measure preserving (in the sense that $\abs{f^{-1}_*(E)} = \abs{E}$ for every measurable $E \subset [0,1]$) and $f_{*}(0) = f_{*}(1) = 0$. Moreover, for each $x, y \in [0,1]$  we have $f_{*}(x) \geq f_{*}(y)$ if and only if $w_{*}(x) \geq w_{*}(y)$.
\end{lemma}
\begin{proof}
 The argument is separated into several steps.
	
	\textbf{Step 1.} \textit{If $f(x) > f(y)$ then $w(x) \geq w(y)$.}
	
	Suppose for the sake of contradiction that $w(x) < w(y)$. Since $w$ is continuous, there exists some $\e > 0$ such that $w(x') < w(y')$ for any $x',y' \in [0,1]$ with $\abs{x'-x}, \abs{y'-y} < \e$. Suppose for concreteness that $x < \kappa < y$ (the remaining cases being either fully analogous or considerably simpler), so that $f$ is increasing in a neighbourhood of $x$ and decreasing in a neighbourhood of $y$. We may assume that $\e$ is small enough that $x + \e < \kappa < y-\e$.
	
	Consider a function $g$ obtained by ``swapping'' the intervals $[x,x+\e]$ and $[y,y+\e]$ in $f$, defined by 
	
	\begin{align*}
		g(x + t) &= f(y + t) && \text{ for all } t \in [0,\e],& \\
		g(y + t) &= f(x + t)  &&\text{ for all } t \in [0,\e],& \\
		g(z) &= f(z) \qquad && \text{ for } z \in [0,x)\cup(x+\e,y)\cup(y+\e,1].& 
	\end{align*}
	
	It is clear that $g \in \cF$, since membership in $\cF$ is ``invariant under rearrangement'' in the sense that the condition $g \in \cF$ can be phrased purely in terms of the values $\abs{g^{-1}(E)}$ for $E \subset [0,1]$, measurable.
	Put $h = g - f$, so that $f + \tau h \in \cF$ for $\tau \in [0,1]$ because of convexity. Using Proposition \ref{EXTup:lem:delta-Lambda} (and the notation therein) we obtain:
\[
	\delta_h \Lambda(f) = \int_0^1 h(z)w(z) dz
	= \int_0^\e ( w(y+t)-w(x + t)) (  f(x+t) - f(y+t)) dt > 0,
\]
	where the last inequality holds since the expression under the integral is strictly positive. Hence, $\Lambda( f + \tau h) > \Lambda(f)$ for sufficiently small $\tau >0$, contradicting the choice of $f$.
	
	\textbf{Step 2.} \textit{The function $f$ is nowhere constant.}
	
	 Suppose for the sake of contradiction that there exists some $y_0 \in [0,1]$ such that $U := f^{-1}(\{y_0\})$ has positive measure, $|U| = s_0 > 0$. Condition \eqref{EXTup:cond:01} implies that $y_0 > 0$. Moreover, if $y_0$ were equal to $1$ then the same condition would imply
	 \[\frac{1}{2} = \int_{0}^1 f(x) dx \geq \frac{1}{2}(1-s)^2 + s = \frac{1}{2} + \frac{1}{2}s^2,\]
	  which is clearly impossible, whence $y_0 < 1$. 
	  
	  Let $h$ be any smooth test function with $\operatorname{supp} h \subset U$, $\norm{h}_{\infty} \leq 1$ and $\int_0^1 h(x) dx = 0 $. (Recall that $\operatorname{supp} h = \operatorname{cl}\{ x \in [0,1] \ : \ h(x) \neq 0 \}$.) 
	 We claim that $f + \tau h \in \cF$ for sufficiently small $\tau > 0$; in fact, it is enough that $\tau \leq \tau_0 := \min(s_0^2/100,y_0,1-y_0)$. It is clear that $0 \leq f + \tau h \leq 1$ (since $\tau \leq y_0,1-y_0$ and $\norm{h}_{\infty} \leq 1$). It remains to show that for any measurable $E \subset [0,1]$ with $\abs{E} > 0$ we have 
\begin{equation}\label{eq:UP:760}
	 \int_{E} f(x) dx + \tau \int_E h(x) dx \geq \frac{1}{2} \abs{E}^2.
\end{equation}
Take any $E$ and put $E_0 = E \setminus U$ and $E_1 = E \cap U$, $t_0 = \abs{E_0}$, $t_1 = \abs{E_1} \in [0,s_0]$. Since $f \in \cF$, for each $s \in [0,s_0]$ we have
\begin{equation}\label{eq:UP:761}
	P(s) := \int_{E_0} f(x) dx + s y_0 - \frac{1}{2}\bra{s+t_0}^2 \geq 0
\end{equation}
The discriminant $\Delta_P$ of the quadratic polynomial $P$ is positive and $\Delta_P \geq \frac{1}{4}s_0^2$ (since $P$ has two real roots roots, which are $\geq s_0$ apart). It is also elementary to verify that $P(s) + \frac{1}{2}P'(s)^2 = \frac{1}{2}\Delta_P$. Define $H(s) = \inf\left\{ \int_F h(x) dx \ : \ \abs{F} = t,\ F \subset U\right\}$ for $s \in [0,s_0]$. Then $H$ is smooth, convex, $H(0) = H(s_0) = 0$, $\norm{H'}_{\infty} \leq 1$ and \eqref{eq:UP:760} will follow once we show that 
\begin{equation}\label{eq:UP:762}
P(s) + \tau H(s) \geq 0 \text{ for all } s \in [0,s_0].
\end{equation}
(in fact, we only need the case $s = t_1$). Let $s_1 \in [0,s_0]$ be a point where $P(s) + \tau H(s)$ takes the minimal value. Since \eqref{eq:UP:762} clearly holds for $s \in \{0,s_0\}$, we may assume that $s_1 \in (0,s_0)$. Then $s_1$ is a local minimum, so $P'(s_1) + \tau H'(s_1) = 0$, whence $\abs{P'(s_1)} \leq \tau$, which in turn implies that
\begin{equation*}\label{eq:UP:763}
P(s_1) + \tau H(s_1) \geq \frac{1}{8} s_0^2 - \frac{3}{2} \tau \geq 0.
\end{equation*}

This finishes the proof that $f+\tau h \in \cF$ for sufficiently small $\tau > 0$. Note that also $f-\tau h \in \cF$, since we may replace $h$ with $-h$. Letting $\tau \to 0$ and using the fact that $f$ maximises $\Lambda$, we conclude that $\delta_h \Lambda(f) = 0$. Hence, by Proposition \ref{EXTup:lem:delta-Lambda} we have $\int_U h(z)w(z) dz = 0$. Since $h$ is arbitrary (subject to the constraint mentioned above) and $w$ is continuous, this is only possible if $w$ is constant on $U$. On the other hand, we also know that $w$ is strictly bitonic, which leads to contradiction and finishes the argument.

\textbf{Step 3.} \textit{For any $t \in [0,1]$, we have $\displaystyle \inf_{\abs{E}=t} \int_{E} f(x) dx = \frac{1}{2} {t^2}$, where the infimum is taken over all measurable sets $E \subset [0,1]$ of the specified measure.}
				
	Note that inequality in one direction follows directly from the fact that $f \in \cF$. It remains to prove that for each $t \in [0,1]$ there exists a measurable set $E \subset [0,1]$ such that $\abs{E} = t$ and $\int_{E} f(x) dx \leq \frac{1}{2}t^2$. 

 	Suppose for the sake of contradiction that $\int_{E_0} f(x) dx > \frac{1}{2}t_0^2$ for some $t_0 \in [0,1]$, where $E_0$ is chosen so that the integral $\int_{E_0} f(x) dx$ is minimised subject to $\abs{E_0} = t_0$. From bitonicity of $f$ it follows that $E_0$ takes the form $E_0 = [0,x_0] \cup [y_0,1]$ for some $x_0,y_0$ such that $0 \leq x_0 \leq \kappa \leq y_0 \leq 1$. Because $f$ is nowhere constant by Step 2, the choice of $x_0,y_0$ is unique.
 	
 	Let $t_1 > t_0$ be such that $\int_{E_0} f(x) dx = \frac{1}{2}t_1^2$, and let $E_1$ be the set minimising $\int_{E_1} f(x) dx$ subject to $\abs{E_1} = t_1$. For the same reasons as above, there are unique $x_1,y_1$ such that $E_1 = [0,x_1] \cup [y_1,1]$. Put $F = E_1 \setminus E_0 = (x_0,x_1] \cup [y_1, y_0)$, and let $U$ be a non-empty open set with $U \subset F$. Since (by Step 2) for any $x \in U$ and $x' \in E_0$ we have $f(x) > f(x')$, one can show using the same techniques as in previous steps that for any smooth test function $h$ with $\operatorname{supp} h \subset U$, $\norm{h}_\infty \leq 1$ and $\int_0^1 h(x)dx = 0$ we have we have $f + \tau h \in \cF$ for all $\tau$ sufficiently small in absolute value. Hence, $\int_0^1 h(x) w(x)dx = 0$, and since $w$ is nowhere constant and $h$ is arbitrary, we reach a contradiction.
 
\textbf{Step 4.} \textit{The function $f$ is measure preserving.} 
	
	Let $g \colon [0,1] \to [0,1]$ be the increasing rearrangement of $f$, i.e., the unique (up to equality almost everywhere) increasing function such that $g^{-1}(E) = f^{-1}(E)$ for any measurable $E \subset [0,1]$. Then $f$ is measure preserving if and only if $g$ is measure preserving. By Step 3, for any $t \geq 0$ we have 
	\[
		\int_{0}^t g(x) dx = \inf_{\abs{E} = t} \int_E f(x) dx = \frac{1}{2} t^2.
	\]
	Differentiating and using the fact that $g$ is increasing, we conclude that $g(x) = x$ for all $x \in [0,1]$; in particular, $g$ is measure preserving.

\textbf{Step 5.} \textit{The function $f$ is continuous.}
	
	Suppose for the sake of contradiction that $f$ were not continuous. Because $f$ is bitonic and chosen so that for each $x \in [0,1]$ we have $f(x) = \limsup_{y \to x} f(y)$, it follows that $f$ has a jump discontinuity at some point $x_0$. Assume without loss of generality that $x_0 \leq \kappa$, and let 
	\[t^+ = f(x_0) = \limsup_{x \to x_0} f(x) >  \liminf_{x \to x_0} f(x) = t^-.\]
Since $f$ is measure preserving by Step 4 and $f$ is bitonic, $f^{-1}((t^-,t^+))$ in an interval of length $t^{+}-t^{-}$ contained in $[\kappa,1]$. Because $w$ is continuous, it follows from Step 1 that $w(y) = w(x_0)$ for all $y \in [0,1]$ with $t^- \leq y \leq t^{+}$. This is a contradiction, since we know that $w$ is nowhere constant.

\textbf{Step 6.} \textit{The lemma holds true.}

	The proof is essentially finished. We have shown that $f$ is continuous and measure preserving in Steps 5 and 4 respectively. We already know that $f$ is bitonic so it is strictly bitonic by Step 2. By Step 1, the condition
\begin{equation}\label{eq:UP:054}
	f(x) \geq f(y) \iff w(x) \geq w(y)
\end{equation}	
	holds for all pairs $(x,y) \in [0,1]^2$ with $f(x) \neq f(y)$. Because $f$ is nowhere constant, \eqref{eq:UP:054} holds for almost all pairs $(x,y) \in [0,1]^2$. Because $f$ and $w$ are continuous, the set of such pairs $(x,y)$ satisfying \eqref{eq:UP:054} is closed. Combining these two facts we conclude that \eqref{eq:UP:054} holds for all $(x,y) \in [0,1]^2$. Finally, $f(0) = f(1) = 0$ follows from the previous considerations because $w(0)=w(1) = 0$.
\end{proof}

\subsection{}
For continuous, nowhere constant $f \in \cFm$ with $f([0,1]) = [0,1]$ and $f(0) = f(1) = 0$, we introduce local inverse functions $\alpha_f \colon [0,1] \to [0,\kappa_f]$ and $\beta_f\colon [0,1] \to [\kappa_f,1]$ so that 
\begin{equation}\label{eq:UP:def-of-a-b}
	f(\a_f(t)) = f(\b_f(t)) = t.
\end{equation}	
	When $f = f_{*}$, the corresponding functions are denoted simply by $\a_{*}$ and $\b_{*}$. Additionally, let $u_{*}$ and $v_{*}$ be the functions defined by \eqref{eq:UP:def-of-u} and \eqref{eq:UP:def-of-v} respectively with $f = f_*$. We have the following, somewhat unexpected, relation.
\newcommand{\til}{\tilde}

\begin{lemma}\label{EXTup:obs:at+bt=cons}
For any $t \in [0,1]$ we have $v_{*} \circ \a_{*}(t) + u_{*} \circ \b_{*}(t) = 1$. In particular, $\kappa_{*} = \frac{1}{2}$.
\end{lemma}
\begin{proof}
	Like before, we omit $*$ in the subscript. 
	 The claim is clearly true for $t = 1$, and the value of $\kappa$ follows from the remaining part of the statement by taking $t = 0$.
	
	It is a direct consequence of the fact that $f$ is measure preserving by Lemma \ref{EXTup:lem:constraints} that $\a$ and $\b$ are Lipschitz continuous with Lipschitz constant at most $1$. Moreover, it follows directly from \eqref{eq:UP:def-of-u} and the fact that $f$ is bitonic that $u$ is Lipschitz continuous on $[0,\kappa-\e]$ for any $\e > 0$ (with Lipschitz constant $\leq 1/f(v(\kappa-\e))$, dependent on $\e$). Likewise, $v$ is Lipschitz continuous on $[\kappa+\e,1]$.
	 Thus, $v \circ \a + u \circ \b$ is Lipschitz continuous on $[0,1 -\e]$ for any $\e > 0$. In particular, $v \circ \a + u \circ \b$ is absolutely continuous on $[0,1 -\e]$ for any $\e > 0$. 
	 
	 Letting $\e \to 0$ we see that to finish the proof it will suffice to show that 
\begin{equation}\label{EXTup:eq:43-z}
	 (v \circ \a + u \circ \b)'(t)=0
\end{equation}
for (Lebesgue-)almost all $t \in [0,1]$. 
A standard computation yields:
	\begin{equation}
		(u \circ \b)'(t) = \b'(t) \frac{ f \circ \b(t) }{f \circ u \circ \b(t)} = \frac{t \b'(t) }{f \circ u \circ \b(t)}, \qquad
		(v \circ \a)'(t) = - \frac{t \a'(t) }{f\circ v \circ \a(t)} \label{EXTup:eq:43-a}
	\end{equation}
	almost everywhere (where $\a'(t)$ and $\b'(t)$ are defined).	
An application of Lemma \ref{EXTup:lem:constraints} gives $w(\a(t)) = w(\b(t))$ (with $w$ defined in Proposition \ref{EXTup:lem:delta-Lambda}). Differentiating this equality leads to
	\begin{equation}
		\frac{\b'(t)}{f \circ u \circ \b(t)} = \frac{\a'(t)}{f\circ v \circ \a(t)} .
	\label{EXTup:eq:43-b}
	\end{equation}
	Combining \eqref{EXTup:eq:43-a} and \eqref{EXTup:eq:43-b} gives \eqref{EXTup:eq:43-z} and finishes the proof.
\end{proof}

	We are now ready to obtain the final piece of information we need about $f_*$, namely the symmetry.
	
\begin{proposition}\label{EXTup:obs:symmetry}
For any $s, t \in [0,1]$, if $\a_{*}(s) = {u_{*} \circ \b_{*}}(t)$ then also $\a_{*}(t) = {u_{*} \circ \b_{*}}(s)$. In particular, the function $f_{*}$ is symmetric: $f_{*}(x) = f_{*}(1-x)$ for all $x \in [0,1]$. 
\end{proposition}
\begin{proof}
	Like before, we omit $*$ in the subscript. 	
	We begin by proving the symmetry of $f$, assuming the former part of the claim. It will be enough to show that $\a(t) + \b(t) = 1$ for any $t \in [0,1]$. Take any $t \in [0,1]$, and let $s = f \circ u \circ \b(t)$, so that $\a(s) = {u \circ \b}(t)$. By assumption, $\a(t) = {u \circ \b}(s)$. We also have ${v \circ \a}(s) = v \circ u \circ \b(t) = \b(t)$. Hence, $\a(t) + \b(t) = {u \circ \b}(s) + {v \circ \a}(s) = 1$ by Lemma \ref{EXTup:obs:at+bt=cons}.
	
	For the remaining part of the argument, it will be convenient to introduce a pair of transformations $T_\a$ and $T_\b$ on $[0,1]$ given by $T_\a = f \circ {u \circ \b}$ and $T_\b = f \circ {v \circ \a}$. With this notation, it will suffice to show that $T_\alpha^2 = \id$, where $T_\a^2 := T_a \circ T_\a$ and $\id$ denotes the identity map on $[0,1]$. Indeed, if $t$ and $s$ are such that $\alpha(s) = u \circ \beta(t)$ then $s = T_\alpha(t)$ and $u \circ \beta (s) = \alpha \circ T^2_\alpha(t)$. We note several properties of these transformations.
	\begin{enumerate}
	\item\label{EXTup:cond:T1} $T_\a \circ T_\b = \id$ and $T_\b \circ T_\a = \id$;
	\item\label{EXTup:cond:T2} $\a \circ T_\a = u \circ \b$ and $\b \circ T_\b = v \circ \a$;
	\item\label{EXTup:cond:T2-a} $v \circ \a \circ T_\a = \b$ and $u \circ \b \circ T_\b = \a$; 
	\item\label{EXTup:cond:T4} $T_\a$ and $T_\beta$ are decreasing;
		
	\item\label{EXTup:cond:T3} ${ \a \circ T_\a^2(t) - \a \circ T_\a^2(t') } = { \b(t') - \b(t)}$  for any $t,t' \in [0,1]$.
	\end{enumerate}
	Assertions \eqref{EXTup:cond:T1}, \eqref{EXTup:cond:T2} and \eqref{EXTup:cond:T2-a} follow directly by substitution. For instance, 
	\[
		T_\a \circ T_\b = f \circ {u \circ \b} \circ f \circ {v \circ \a}
		=  f \circ {u} \circ {v \circ \a} =   f \circ{ \a} = t,
	\]
	where we use the facts that $\beta \circ f \circ v = v$, $u \circ v \circ \alpha = \alpha$ and $f \circ \alpha = \id$. The remaining equalities follow along similar lines.
	
	Assertion \eqref{EXTup:cond:T4} follows from known monotonicity properties of $f$, $\a$, $\b$ and $u$, $v$. If $t < t'$ then $\b(t) > \b(t')$, hence $u \circ \beta(t) > u \circ \beta(t')$ and $f\circ u \circ \beta(t) > f \circ u \circ \beta(t')$ (note that $f$ is strictly increasing in the relevant range). Hence, $T_\a$ is decreasing, and $T_\b$ is decreasing by \eqref{EXTup:cond:T1}.
	
	Assertion \eqref{EXTup:cond:T3} follows from Lemma \ref{EXTup:obs:at+bt=cons} and previously established properties, since for any $t,t' \in [0,1]$ we have
	
	\begin{align*}
	{ \a \circ T_\a^2(t) - \a \circ T_\a^2(t') } & \overset{\eqref{EXTup:cond:T2}}{=}
	{ u \circ \b \circ T_\a(t) - u \circ \b \circ T_\a(t') } 
	\\ &\overset{\eqref{EXTup:obs:at+bt=cons}}{=} { v \circ \a \circ T_\a(t') - v \circ \a \circ T_\a(t) }
	 \overset{\eqref{EXTup:cond:T2-a}}{=} { \b (t') - \b(t) },
	\end{align*}	
	
	 Suppose for the sake of contradiction that for some $t_0$ we have $T_\a^2(t_0) \neq t_0$. For concreteness, we may suppose that $T_\a^2(t_0) > t_0$, the other case being fully analogous (it is enough to run the same argument with $\a$ and $\b$ interchanged). 	Let us consider the consecutive iterates $t_n := T_\a^n(t_0)$ for $n \in \ZZ$ (for $n < 0$ we use $T_\beta = T_\a^{-1}$). 
	 
	It follows from \eqref{EXTup:cond:T4} and a straightforward induction that the sequence $(t_{2n})_{n \in \ZZ}$ is strictly increasing and $(t_{2n+1})_{n \in \ZZ}$ is strictly decreasing
	
	As a consequence of \eqref{EXTup:cond:T2}, for any $t,t' \in [0,1]$ we have 
\begin{equation}\label{eq:UP:734}
	0 = \int_{\a(t)}^{v\circ \a(t)}f(x)dx - \int_{\a(t')}^{v\circ \a(t')} f(x)dx 
	=  \int_{\a(t)}^{\a(t')}f(x)dx - \int_{\b \circ T_\b(t)}^{\b \circ T_\b(t')} f(x)dx,
\end{equation}
where we use the previously introduced convention that $\int_{a}^{b} = -\int_{b}^{a}$.
In particular, applying \eqref{eq:UP:734} with $t=t_n$ and $t' = t_{n+2}$ we may define
	\begin{equation}\label{eq:UP:def-of-I}
		I_n := { \int_{\a(t_n)}^{\a(t_{n+2})} f(x) dx } = { \int_{\b(t_{n-1})}^{\b(t_{n+1})} f(x) dx }.
	\end{equation}
Similarly, applying \eqref{EXTup:cond:T3} $t=t_{n-2}$ and $t' = t_{n}$ may define
\begin{equation}\label{eq:UP:def-of-l}
		l_n := {\a(t_{n+2}) - \a(t_{n}) } =  {\b(t_{n-2}) - \b(t_{n}) } > 0.
\end{equation}

Estimating the integrals in \eqref{eq:UP:def-of-I} and exploiting known monotonicity properties of $f$ and $t_n$ we obtain for any $n \in \ZZ$ the following system of inequalities: 
	
	\begin{align}
	l_{2n} t_{2n+2}  &\geq I_{2n} \geq l_{2n} t_{2n}, \label{eq:UP:131}
	&l_{2n+1} t_{2n+1}  &\geq I_{2n+1} \geq l_{2n+1} t_{2n+3}, \\
	l_{2n+1} t_{2n-1}  &\geq I_{2n} \geq l_{2n+1} t_{2n+1}, \label{eq:UP:132}
	&l_{2n+2} t_{2n+2}  &\geq I_{2n+1} \geq l_{2n+2} t_{2n}.
	\end{align}
Combining the inequalities in \eqref{eq:UP:131} and \eqref{eq:UP:132} in such a way as to eliminate the appearances of $I_n$ we obtain estimates for the ratios $l_n/l_{n+1}$; namely, for any $n \in \ZZ$,

	\begin{align}\label{eq:UP:140}
\frac{t_{2n-1}}{t_{2n}} &\geq \frac{l_{2n}}{l_{2n+1}} \geq \frac{t_{2n+1}}{t_{2n+2}},
& \frac{t_{2n+2}}{t_{2n+3}} &\geq \frac{l_{2n+1}}{l_{2n+2}} \geq \frac{t_{2n}}{t_{2n+1}}.
	\end{align}
From \eqref{eq:UP:140} we can in turn obtain bounds for the ratios $l_n/l_{n+2}$, namely

	\begin{align}\label{eq:UP:141}
	\frac{t_{2n-1}t_{2n+2}}{t_{2n}t_{2n+3}} &\geq \frac{l_{2n}}{l_{2n+2}} \geq \frac{t_{2n}}{t_{2n+2}},
& \frac{t_{2n-1}}{t_{2n+1}} \geq \frac{l_{2n-1}}{l_{2n+1}} \geq \frac{t_{2n-2}t_{2n+1}}{t_{2n-1}t_{2n+2}}.
	\end{align}	
In particular, $l_{2n}/t_{2n}$ is decreasing and $l_{2n+1}/t_{2n+1}$ is increasing with $n$. Hence, $l_{2n+1}/t_{2n+1}$ converges to a strictly positive limit (possibly equal to $+\infty$) as $n \to \infty$, while $l_{2n+1} \to 0$ (in fact, $\sum_{n = -\infty}^{n=+\infty} l_n < \infty$). It follows that 
	\begin{equation}\label{eq:UP:151}
	t_{2n+1} \to 0 \text{ as } n \to \infty.
	\end{equation}	
Following a symmetric argument, we also conclude that 
	\begin{equation}\label{eq:UP:152}
	t_{2n} \to 0 \text{ as } n \to -\infty.
	\end{equation}	

Relations \eqref{eq:UP:151} and \eqref{eq:UP:152} are impossible to reconcile with the monotonicity properties of $t_n$. To see this, note fist that there is no $n \in \ZZ$ such that $t_{n+2}$ lies between $t_{n}$ and $t_{n+1}$ (here and elsewhere, $x$ lies between $y$ and $z$ if either $y < x < z$ or $z < x < y$); indeed, if such $n$ existed then as simple inductive argument would show that $t_m$ lies between $t_{n}$ and $t_{n+1}$ for all $m \geq n+2$, which contradicts \eqref{eq:UP:151}. Secondly, there is no $n \in \ZZ$ such that $t_{n}$ lies between $t_{n+1}$ and $t_{n+2}$, since that would lead to a contradiction with \eqref{eq:UP:152}. Thus, for any $n \in \ZZ$, either $t_{n} < t_{n+1} < t_{n+2}$ or $t_{n} > t_{n+1} > t_{n+2}$. By induction of $n$ we conclude that $(t_{n})_{n \in \ZZ}$ is a monotonous sequence. This contradicts the previous observation that $t_0 < t_2$ while $t_1 > t_3$, and finishes the argument.
\end{proof}

\subsection{} 
The proof of the main result is now essentially finished.

\begin{corollary}\label{cor:UP:final}
	The function $f_{*} \in \cFm$ is the unique (up to equality a.e.) function maximising $\Lambda$ and is given by:
	\[
		f_{*}(x) = 
		\begin{cases}
			2x &  \text{ if } x \leq \frac{1}{2},\\
			2(1-x) &  \text{ if } x \geq \frac{1}{2}.
		\end{cases}
	\]
\end{corollary}
\begin{proof}
	For any $t \in [0,1]$ we have $\alpha_{*}(t) = 1-\beta_{*}(t)$ by Proposition \ref{EXTup:obs:symmetry} and $\alpha_*(t) + 1 - \beta_{*}(t) = t$ by Lemma \ref{EXTup:lem:constraints}. Hence, $\alpha_*(t) = t/2$ and $\beta_{*}(t) = 1-t/2$, meaning that $f(t/2) = f(1-t/2) = t$.
\end{proof}

\begin{proof}[Proof of Proposition \ref{EXTup:prop:L(f)-bound}]
The region ${L}(f_{*})$ is the quarter-circle given by 
\[{L}(f_{*}) = \bbra{ (x, y) \in [0,1]^2 \ : \ x^2 + (1-y)^2 \leq \frac{1}{4} }.\] 
This allows us to compute that $\Lambda(f_{*}) = \frac{\pi}{16}$, which is precisely the stated bound.
\end{proof}
 
\section{Closing remarks}\label{section:END}

\subsection{}
In the previous sections, we have obtained a fairly satisfactory understanding of $\abs{S(a)}$ for a randomly selected permutation $a$, as well as in the ``best case scenario'' where $a$ is chosen to maximize $\abs{S(a)}$. It is natural to also ask about the behaviour of $\abs{S(a)}$ in the ``worst case scenario'', when $a$ is chosen so as to minimize $\abs{S(a)}$. We now address this problem, but we ask more questions than we answer.

The best lower bound we are aware of can be obtained by an argument in \cite{Solymosi-2005} (also present in \cite{Balog-2014}), a variant of which we sketch below for the convenience of the reader.

\begin{proposition}\label{prop:END:001}
	For any $n \geq 1$ and any $a \in \Sym([n])$ it holds that $\abs{S(a)} \geq \ffrac{n^{3/2}}{4\sqrt{2}}$.
\end{proposition}

\begin{proof}
	For any integer $k$, consider the set $S^k(a)$ of the sums $s \in S(a)$ with $kn < s \leq (k+1)n$. Clearly, $S(a) = \bigcup_{k = 0}^{\infty} S^k(a)$, and the union is disjoint.
	
	Take any $k$ with $1 \leq k \leq \ffrac{(n+1)}{4}$. For any $u$ with $1 \leq u \leq n$, at least one of the sums $\sum_{i=u}^n a_i$ or $\sum_{i=1}^u a_i$ exceeds $ \frac{1}{2} \binom{n+1}{2} \geq k n $. For concreteness, suppose $\sum_{i=u}^n a_i > kn$ and let $v \leq n+1$ be the smallest integer such that $\sum_{i=u}^{v-1} a_i > k n$; the remaining case is entirely analogous. 
	
	By the choice of $v$, we have $\sum_{i=u}^{v-1} a_i \in S^{k}(a)$ and $\sum_{i=u+1}^{v-1} a_i  \in S^{k}(a)\cup S^{k-1}(a)$. Note that for any set $R \subset \NN$ we have $\abs{(R-R) \cap \NN} \leq \abs{R}^2/2$. Applying this with $R = S^k(a) \cup S^{k-1}(a)$ and recalling that the choice of $u$ with $1 \leq u \leq n$ was arbitrary, we conclude that 
	\begin{equation}\label{eq:END:001}
	\abs{S^k(a)} + \abs{S^{k-1}(a)} \geq \sqrt{2n}.
	\end{equation}
	Summing \eqref{eq:END:001} over $1 \leq k \leq \frac{n+1}{4}$ and using the fact that $\abs{S^0(a)} = n$ we obtain
	\[
	\abs{S(a)} 
	\geq \frac{n}{2} + \frac{1}{2} \sum_{k=1}^{\floor{\frac{n+1}4}} \bra{ \abs{S^k(a)} + \abs{S^{k-1}(a)}} 
	\geq \frac{n}{2} + \frac{1}{2} \floor{\frac{n+1}{4}} \sqrt{2n} \geq \frac{n^{3/2}}{4\sqrt{2}}.\qedhere 
	\]
\end{proof}

\subsection{}

In contrast to Proposition \ref{prop:END:001} essentially the best available upper bound for the least possible size of $\abs{S(n)}$ for $a \in \Sym([n])$ corresponds to the trivial permutation, for which we have 
\[\abs{S(\id_n)} = n^{2-o(1)}.\]
Slightly more generally, we have a similar result for permutations of ``bounded complexity''. For an integer $M \geq 1$, let us say that a permutation $a \in \Sym([n])$ has complexity at most $ M$ if there exists a partition $[n] = \bigcup_{j=1}^M I_j$ into disjoint intervals $I_j$ such that for each $1 \leq j \leq M$ there exist integers $b_j,c_j$ with $\abs{b_j} \leq M$  such that $a_i = i b_j + c_j$ and each $i \in I_j$. Because the result is rather standard, we only sketch the argument.

\begin{example}\label{exple:END:bounded-complexity}
	Fix an integer $M$. Let $a \in \Sym([n])$ be a permutation with complexity $\leq M$. 	
	Then
	\begin{equation}\label{eq:END:002}
	\abs{S(a)} = o(n^2) \text{ and} \abs{S(a)} = n^{2-o(1)},
	\end{equation} 
	where he implicit rates of convergence are allowed to depend on $M$.
\end{example}
\begin{proof}[Sketch of the proof]
	Pick a partition $[n] = \bigcup_{j=1}^M I_j$ and integers $b_j, c_j$ ($1 \leq j \leq M$) as in the definition of complexity above.
	
	For the lower bound, it is enough to notice that one of the intervals $I_j$ ($1 \leq j \leq M$) has length $\geq n/M$. For any $u,v$ with $\min I_j \leq u < v \leq \max I_j+1$, $S(a)$ contains the sum
	\begin{equation}\label{eq:END:202}
	  s = \sum_{i=u}^{v-1} a_i = (v-u)\bra{c_j + (u+v-1)b_j/2 }.
	\end{equation}
	There are at least $n^2/2M^2$ choices of $u,v$ as above. Conversely, for each sum $s \in S(a|_{I_j})$, each pair $u,v$ such that \eqref{eq:END:202} holds corresponds uniquely to a decomposition of $s$ as a product of two integers, whence there are at most $d(s)$ such pairs $u,v$, where $d(s)$ is the divisor function. Since it is a well known fact that $\max_{s \leq n} d(s) =  n^{o(1)}$, we conclude that $\abs{S(a)} \geq n^{2-o(1)}$.
	
	For the upper bound, it suffices to show that for any $j,k$ with $1 \leq j,k \leq M$ the number of distinct sums $\sum_{i=u}^{v-1} a_i$ with $u \in I_j$ and $v-1 \in I_k$ is $o(n^2)$. Fix a choice of $j,k$ and put $u_0 = \max I_j+1$, $v_0 = \min I_k$, $u = u_0 - x$, $v = v_0 + y$ then 
	\begin{equation}\label{eq:END:004}
	 \sum_{i=u}^{v-1} a_i = Ax^2 + By^2 + Cx + Dy + E,
	\end{equation} 
where $A = \ffrac{b_j}{2}$, $B = \ffrac{b_k}{2}$, and $C,D,E$ are some constants with $C,D = O_M(n)$. Hence, the problem reduces to showing that for any polynomial $P(x,y) = Ax^2 + By^2 + Cx + Dy + E$ subject to the above constraints we have
\begin{equation}\label{eq:END:003}
	\abs{\bbra{ P(x,y) \ : \ 1 \leq x,y \leq n }} = o(n^2).
	\end{equation}
It is not difficult to reduce further to the case $C = D = E = 0$, $A,B$ are integers and $\gcd(A,B) = 1$. There are now two cases to consider, depending on whether $-AB$ is a square of an integer.

If $-AB$ is a square then $P(x,y)$ factors as $(A'x + B'y)(A'x -B'y)$ and the bound follows from the theorem of Erd\H{o}s on the multiplication table problem cited in the introduction.

Suppose now that $-AB$ is not a square. By Chebotarev's density theorem, there exists a family $\cP_0$ of primes with positive relative density in the set of all primes $\cP$ (i.e., $\liminf_{N \to \infty} \ffrac{\abs{\cP_0 \cap [0,N)}}{\abs{\cP \cap [0,N)}} > 0$) such that $-AB$ is not a square modulo $p$ for any $p \in \cP_0$. For any $p \in \cP_0$ we then have \[P(x,y) \not \equiv p,2p,\dots,(p-1) p \pmod{p^2}, \quad \text{ for all } x, y \in \ZZ.\]
Letting $w(n)$ be a sufficiently slowly increasing function and putting $K = \abs{A}+\abs{B}$ we conclude that

\begin{align*}
\abs{\bbra{ P(x,y) \ : \  x,y \in [n] }} &\leq \abs{\bbra{ P(x,y) \ : \ x, y \in \ZZ } \cap [-K n^2, K n^2] }  
\\  &\leq (2K+1) n^2 \prod_{\substack{ p \in \cP_0 \\ p \leq w(n) }} \bra{1-\frac{p-1}{p^2}} 
+ 2 \prod_{\substack{ p \in \cP_0 \\ p \leq w(n) }} p^2.
\end{align*}
As long $w(n) \to \infty$ and $w(n) = o(\log n)$, both summands above are $o(n^2)$.
\end{proof}

\subsection{}
Example \ref{exple:END:bounded-complexity} and Theorem \ref{thm:INT:B} show that the estimate $\abs{S(a)} = n^{2-o(1)}$ holds both for generic and highly structured permutations $a \in \Sym([n])$. This prompts the following question.

\begin{question}
	Is it true that $\min_{a \in \Sym([n])} \abs{ S(n)} = n^{2-o(1)}$? That is, is it true that for any $\delta > 0$ there exists $c_\delta > 0$ such that for any $n \geq 1$ the bound $\abs{S(a)} \geq c_\delta n^{2-\delta}$ holds for all $a \in \Sym([n])$?
\end{question}

In fact, all examples of $a \in \Sym([n])$ with $\abs{S(a)} = o(n^2)$ we are aware of exhibit some algebraic structure, much as in Example \ref{exple:END:bounded-complexity}. It is not the case that $\abs{S(a)}$ is minimised for the trivial permutation $\id_n$, but none of the examples known to the author are significantly worse. Hence, we ask a bolder question.

\begin{question}
	Does there exist an absolute constant $c > 0$ such that for any $n \geq 1$ the bound $\abs{S(a)} \geq c \abs{S(\mathrm{id}_n)}$ holds for all $a \in \Sym([n])$?
\end{question}

\subsection{} In a similar spirit, we may also ask if the only way for $S(a)$ to be small is if $a$ has some algebraic structure. To give an indication of just how much structure one may hope to find, we give the following examples. We omit the proofs, which are rather standard.

\begin{example}
	Fix a constant $M$, let $m \geq 1$ be an integer, and put $n = Mm$. Consider the permutation 
	\[a = (1,m+1,\dots,(M-1)m+1, 2,m+2,\dots,(M-1)m+2,3,\dots),\] that is, $a_{kM+l} = m(l-1) + k + 1$ where $0 \leq k < m$ and $1 \leq l \leq M$. Then $\abs{S(a)} = o(n^2)$ (the speed of convergence depends on $M$). 
\end{example}

Recall, however, that a similar-looking permutation considered in Proposition \ref{I:prop:cexple-basic} has at least $n^2/4$ distinct consecutive sums.

\begin{example}
	Take any permutation $a$ with $\abs{S(a)} = o(n^2)$, and let $m = o(n)$. Consider a permutation $b$ obtained from $a$ by choosing $m$ pairs of consecutive indices $i,i+1$ and swapping $a_i$ with $a_{i+1}$. More precisely, pick any set $I$ with $\abs{I} = m$ and such that $i \in I$ implies $i+1 \not \in I$, and define:
\[
	b_i = 
	\begin{cases}
		a_{i+1} & \text{ if } i \in I,\\
		a_{i-1} & \text{ if } i+1 \in I,\\
		a_{i} & \text{ otherwise}.\\		
	\end{cases}
\]
Then $\abs{S(b)} = \abs{S(a)} + O(mn) = o(n^2)$.
\end{example}

The examples above eliminate several conjectures one could make about permutations with few distinct consecutive sums. For instance, one cannot hope to show that such a permutation agrees with an affine sequence on an interval of length comparable to $n$. We can, however, hope that the following should have a positive answer. Recall that $\omega(1)$ denotes a quantity which tends to $\infty$ as $n \to \infty$.

\begin{question}
	Does there exist $\e > 0$ such that the following is true?
Let $n \geq 1$ and $a \in \Sym([n])$ with $\abs{S(a)} \leq \e n^2$. Then there exists an index set $I \subset [n]$ with $\abs{I} = \omega(1)$ and integers $b,c$ such that $a_i = b i + c$ for all $i \in I$.
\end{question}

\subsection{}
One may also ask a similar questions in a more general context. Let $a = (a_i)_{i=1}^n$ be an ordering of a set $A \subset \NN$ of size $\abs{A} = n$, not necessarily equal to $[n]$. Let $S(a)$ be, as introduced before, the set of sums of consecutive terms of $a$, that is
\[S(a) = \bbra{ \sum_{i=u}^{v-1} a_i \ : \ 1 \leq u < v \leq n+1}.\] How small can $\abs{S(a)}$ be?

It is perhaps more natural to phrase this question in different terms. For a set $B = \{b_i \ : \ 1 \leq i\leq m \}$ with $b_1 < b_2 < \dots < b_m$, define (following the terminology of \cite{Solymosi-2005} and \cite{Balog-2014}) the set of gaps 
\[D(B) = \{ b_{i+1} - b_{i} \ : \ 1 \leq i < m\}.\]
Note that setting 
$B = \left\{ \sum_{i=1}^{v-1} a_i \ : \ 1 < v \leq n+1 \right\} \cup \{0\}$ we can recover $S(a) = (B-B) \cap \NN$ and $D(B) = \{a_i \ : \ 1 \leq i \leq n \} = A$. 

\begin{question} For which $\delta > 0$ does there exist $c_\delta > 0$ such that the following holds?
	Let $n \geq 1$ and let $B \subset \NN$ with $\abs{B} = n$ and $\abs{D(B)} = n-1$.  Then $\abs{B-B} \geq c_\delta n^{2-\delta}$.
\end{question}

This question is already alluded to in \cite{Solymosi-2005}, and resolved positively in the case $\delta  = \frac{1}{2}$. 
However, for $\delta \in (0,\frac{1}{2})$, to the best of our knowledge, the answer is not known.  
\bibliographystyle{alphaabbr}

\bibliography{9-bibliography}

\end{document}